\documentclass{amsart}
\usepackage{amssymb, colordvi}
\usepackage{multirow}
\usepackage[pdftex]{graphicx,color}
\usepackage{todonotes} 
\usepackage{xfrac, faktor} 
\usepackage{mathtools} 
\usetikzlibrary{arrows,positioning,matrix,decorations.pathmorphing,calc,decorations.markings,decorations.pathreplacing,shapes} 

\numberwithin{equation}{section}
\newtheorem{thm}{Theorem}
\numberwithin{thm}{section}
\newtheorem{prop}[thm]{Proposition}
\newtheorem{lemma}[thm]{Lemma}

\theoremstyle{definition}

\newtheorem{example}[thm]{Example}
\newtheorem{remark}[thm]{Remark}
\newtheorem{definition}[thm]{Definition}

\newcounter{FNC}[page]
\def\fauxfootnote#1{{\addtocounter{FNC}{2}$^\fnsymbol{FNC}$%
     \let\thefootnote\relax\footnotetext{$^\fnsymbol{FNC}$#1}}}

\newcommand{\C}{\mathbb{C}}
\newcommand{\N}{\mathbb{N}}

\newcommand{\R}{\mathbb{R}}
\newcommand{\Z}{\mathbb{Z}}

\newcommand{\A}{{\mathcal A}}

\newcommand{\vol}{{\rm vol}}

\renewcommand{\>}{{\rangle}}

\newcommand{\sgnvar}{{\operatorname{sgnvar}}}
\newcommand{\sgn}{{\operatorname{sgn}}}

\title[Optimal Descartes' Rule of Signs for Circuits]{Optimal Descartes' Rule of Signs \\ for Systems supported on Circuits} 

\author{Fr\'ed\'eric Bihan}
\address{Laboratoire de Math\'ematiques\\
         Universit\'e de Savoie\\
         73376 Le Bourget-du-Lac Cedex\\
         France}

\email{Frederic.Bihan@univ-savoie.fr}
\urladdr{http://www.lama.univ-savoie.fr/~bihan/}

\author{Alicia Dickenstein}
\address{Dto.\ de Matem\'atica, FCEN, Universidad de Buenos Aires, and IMAS (UBA-CONICET), Ciudad Universitaria, Pab.\ I, 
C1428EGA Buenos Aires, Argentina}
\email{alidick@dm.uba.ar}
\urladdr{http://mate.dm.uba.ar/~alidick}

\author{Jens Forsg{\aa}rd}
\address{University of Texas at San Antonio, Department of Mathematics, One UTSA Circle, San Antonio, TX 78249, USA}
\email{jensforsgard@gmail.com}

\thanks{AD was partially supported by UBACYT 20020100100242, CONICET PIP 20110100580, and ANPCyT PICT 2013-1110, Argentina. FB was partially supported by grant ANR-18-CE40-0009 ``ENUMGEOM'' of Agence Nationale de Recherche.
JF was funded by SNSF grant \#159240 ``Topics in tropical and real geometry,'' the NCCR SwissMAP project, and NWO-grant TOP1EW.15.313.}  

\begin{document}

\begin{abstract}
We present an optimal  version of Descartes' rule of signs to bound the number
of positive real roots of a sparse system of polynomial equations in $n$ variables with $n+2$ monomials.
This sharp upper bound is given in terms of the sign variation of a sequence associated to the exponents and the
coefficients of the system.
\end{abstract}

\maketitle

\section{Introduction}
Given a  univariate real polynomial  $f = c_0 + c_1 x + \cdots + c_r x^r$, 
 Descartes' rule of signs, proposed in 1637, asserts that the number of 
positive roots of $f$, counted with multiplicity, does not exceed the sign variation $\sgnvar(c_0, \dots, c_r)$ 
of the ordered sequence of its coefficients.  This integer
 $\sgnvar(c_0, \dots, c_r)$ is defined as the number of distinct pairs $(i,j)$ of integers 
between $0$ and $r$ which satisfy that $c_i \cdot c_j <0$ and $c_\ell=0$ for any index
$\ell$ with $i<\ell<j$. This bound is extremely simple and it is sharp in the following sense:
if the signs of the coefficients are fixed in $\{0,\pm1\}$ (with the convention that $0$ has sign $0$),
then one can choose absolute values of the coefficients such that the bound is attained.

Little is known in the multivariate case. We refer the reader to the article~\cite{BD}, where a multivariate
rule of signs is obtained in the following setting.
Fix an ordered exponent set $\A=\{a_0,a_1,\ldots,a_{n+1}\} \subset \Z^n$ of cardinality
$n+2$. For any coefficient matrix $C =(c_{i,j}) \in \R^{n \times (n+2)}$, consider the 
sparse polynomial system in $n$ variables $x=(x_1, \dots, x_n)$ with support $\A$:
\begin{equation}\label{E:system}
f_i(x)=\sum_{j=0}^{n+1} c_{i,j}\,x^{a_j} = 0 \, , \quad i=1,\ldots.
\end{equation}

We denote by $n_\A(C)$ the number of positive solutions  of ~\eqref{E:system}
counted with multiplicity.  We will make throughout the text the natural hypotheses that the rank of $C$ equals $n$ and that the convex hull of $A$ is 
of dimension $n$. Under these assumptions,  Proposition 2.12 in~\cite{BD} gives a checkable condition for $n_\A(C)$ to be finite. 
The problem is to find a sharp upper bound for  $n_\A(C)$, when it is finite, in the spirit
of Descartes' rule of signs.  Such an upper bound was given in Theorem~2.9 in~\cite{BD}; however, this bound is not sharp for any given support $\A$ as in the classical Descartes' rule of signs, see  \cite[Example 5.2]{BD} for instance. We provide an upper bound which is sharp for any given support $\A$ in Theorem~\ref{thm:refines}.

 A first question is under which transformations of $\A$ and $C$ is $n_\A(C)$ invariant.
If we multiply the coefficient matrix $C$ on the left by an invertible $n \times n$ real  matrix $M$, we get an equivalent system and, thus, $n_\A(C)=n_\A(M \, C)$. Therefore, this quantity only depends on the linear span of the rows of $C$ and thus a sharp upper bound should depend on the Pl\"ucker coordinates of this subspace, which are given by the maximal minors of $C$.  
A second observation is that if we multiply all the polynomials by a monomial $x^b$, then we do not change the set of solutions in the torus $(\C^*)^n$. Moreover,  the number of these solutions (and also the number of positive solutions) is invariant under an invertible monomial change of variables in the torus.  Thus, $n_\A(C)$ is an affine invariant of $\A$. We can express this invariance considering the matrix 
$A \in \Z^{(n+1)\times (n+2)}$ whose columns are the vectors $(1, a_j) \in \Z^{n+1}$ for all $a_j \in \A$:
\begin{equation}\label{eq:A}
A=\left[
\begin{array}{ccc}
1 &  \cdots &1 \\
 a_0& \cdots   & a_{n+1} 
\end{array}
\right].
\end{equation}
So, $n_\A(C)$ should also depend on the maximal minors of $A$. Note that the
hypothesis that the convex hull of $\A$  is of dimension $n$ is equivalent to the fact that 
the matrix $A$ has maximal rank $n+1$.  

In fact, $n_\A(C)$ is also invariant after a renumbering of $\A$,
if the same permutation is applied to the columns of $C$.  This 
transformation could only change in the same way the signs of the minors
of $A$ and $C$.

Theorem~2.9 in~\cite{BD} gives an upper bound  in terms of the sign variation of a sequence that depends on
sums of minors of $A$ prescribed by signs of minors of $C$.
 One interesting consequence is Theorem~3.3 in~\cite{BD} which shows that the combinatorics of
the affine configuration of the exponents restricts the maximum possible value of $n_\A(C)$ for any $C$. Consider for instance the case $n=2$. When $\A$ is a circuit, that is, when we
have four points such that no three of them lie on a line,  there are two possible combinatorial configurations, depicted in Figure~\ref{fig:dim2}: either the four points are the vertices of a quadrilateral or one point lies inside the convex hull of the others.
 In this second case (or if $\A$ is not a circuit but consists of four points with three of them on a line), we have that $n_\A(C) \le 2$ for any $C$ for which $n_\A(C)$ is finite, 
 while for circuits in the first case the upper bound might be $3$ depending on $C$.
Moreover, $n_\A(C) \le 2$ in any dimension $n$ when one of the points in the configuration $\A$ lies inside the convex hull of the remaining $n+1$ points.

\begin{figure}[h!]
\definecolor{zzttqq}{rgb}{0.8,0.2,0.}
\definecolor{qqqqff}{rgb}{0.,0.,0.5}
\begin{tikzpicture}[line cap=round,line join=round,>=triangle 45,x=1.0cm,y=1.0cm]
\fill[color=zzttqq,fill=zzttqq,fill opacity=0.1] (-2.02,4.56) -- (-3.22,3.24) -- (1.08,3.38) -- (-0.98,5.32) -- cycle;
\fill[color=zzttqq,fill=zzttqq,fill opacity=0.1] (4.62,5.36) -- (3.34,3.4) -- (6.28,3.46) -- cycle;
\draw [color=zzttqq] (-2.02,4.56)-- (-3.22,3.24);
\draw [color=zzttqq] (-3.22,3.24)-- (1.08,3.38);
\draw [color=zzttqq] (1.08,3.38)-- (-0.98,5.32);
\draw [color=zzttqq] (-0.98,5.32)-- (-2.02,4.56);
\draw [color=zzttqq] (4.62,5.36)-- (3.34,3.4);
\draw [color=zzttqq] (3.34,3.4)-- (6.28,3.46);
\draw [color=zzttqq] (6.28,3.46)-- (4.62,5.36);
\begin{scriptsize}
\draw [fill=qqqqff] (-2.02,4.56) circle (1.5pt);
\draw[color=qqqqff] (-1.84,4.9) node {$a_1$};
\draw [fill=qqqqff] (-3.22,3.24) circle (1.5pt);
\draw[color=qqqqff] (-3.04,3.58) node {$a_0$};
\draw [fill=qqqqff] (1.08,3.38) circle (1.5pt);
\draw[color=qqqqff] (1.26,3.72) node {$a_3$};
\draw [fill=qqqqff] (-0.98,5.32) circle (1.5pt);
\draw[color=qqqqff] (-0.8,5.66) node {$a_2$};
\draw [fill=qqqqff] (4.62,5.36) circle (1.5pt);
\draw[color=qqqqff] (4.8,5.7) node {$a_1$};
\draw [fill=qqqqff] (3.34,3.4) circle (1.5pt);
\draw[color=qqqqff] (3.52,3.74) node {$a_2$};
\draw [fill=qqqqff] (6.28,3.46) circle (1.5pt);
\draw[color=qqqqff] (6.46,3.8) node {$a_3$};
\draw [fill=qqqqff] (4.76,4.26) circle (1.5pt);
\draw[color=qqqqff] (4.94,4.6) node {$a_0$};
\end{scriptsize}
\end{tikzpicture}
\label{fig:dim2}
\caption{The two possible combinatorial circuits in the plane}
\end{figure}
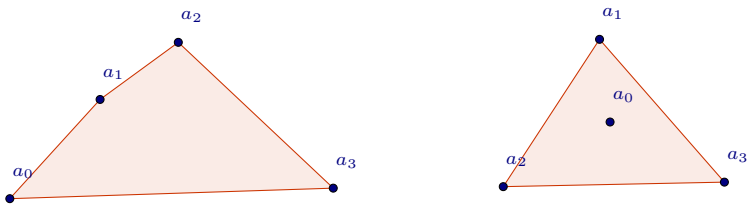

Another observation is that solutions in the torus do not change if we multiply each polynomial by a possibly {\em different} monomial.
To understand this former invariance of $n_\A(C)$, 
consider the associated Cayley matrix
$A \ast \dots \ast A$ in $\Z^{2n \times (n+2) n}$:
\begin{equation}\label{Adecomp} 
A \ast \dots \ast A \ =\ 
\begin{pmatrix}
1 \cdots 1 & 0 \cdots 0 & \cdots & 0 \cdots 0\\
0 \dots 0 & 1 \dots 1 & \cdots & 0 \dots 0\\
\vdots & \vdots& \vdots&\vdots\\
0 \dots 0 & 0 \dots 0 & \cdots & 1 \dots 1 \\
 A'& A' & \cdots & A'
\end{pmatrix},
\end{equation}
where $A' \in \Z^{n\times (n+2)}$ is the matrix with columns $ a_0, \dots,  a_{n+1}$.
So, $n_\A(C)$ is also invariant under left multiplication of $A \ast \dots \ast A$ 
by an invertible matrix and our bound should then depend on the maximal minors of
this Cayley matrix. Note that multiplication of each equation by a possibly different monomial and performing 
the same invertible monomial change of coordinates for all equations corresponds to left multiplication of $A \ast \dots \ast A$  by an invertible matrix.
By a Laplace expansion starting with the upper $n$ rows, we see that all maximal minors
of $A$ occur as appropiate maximal minors of $A \ast \dots \ast A$ and that
any maximal minor of $A \ast \dots \ast A$ is a linear combination with coefficients in $\{0, 1, -1\}$
of maximal minors of $A$.  

Our main result Theorem~\ref{thm:refines} gives a Descartes' rule
of signs in terms of the sign variation of a sequence of minors of the Cayley matrix, defined and ordered
according to the maximal minors of $C$.  Section~\ref{sec:1main} is devoted to present and prove this result.
We show in Lemma~\ref{prop:VolumeBound} that our bound does not exceed
the normalized volume $\vol_{\Z \A}(\A)$, which is a bound for the number of isolated
positive solutions coming from Bernstein-Kouchnirenko theorem.

We present here weaker versions of our results that are easier to state.
For this, consider any matrix $B \in \Z^{(n+2) \times 1}$ which is a Gale
dual of $A$. This means that the column
vector $(b_0, \dots, b_{n+1})^\top$  of $B$ is a basis of the kernel of $A$.
So, up to a multiplicative constant,
\begin{equation}\label{eq:ComplementaryMinors}
b_j = (-1)^j \, \det(A(j)),
\end{equation}
where $A(j)$ denotes the square matrix obtained by deleting from $A$
the $j$th column.
Let  $ \mathfrak{S}_{n+2}$ denote all permutations of $\{0,\dots, n+1\}$.
For any $\sigma \in \mathfrak{S}_{n+2}$, let $b_\sigma=(b_{\sigma_{0}}, \dots b_{\sigma_{n+1}})$ be the sequence
which consists of the coefficients of $B$ ordered with respect to $\sigma$.
Then, \cite[Theorem~2.9]{BD} implies that when $C$ is a uniform matrix (i.e. all maximal minors of $C$ are nonzero)
the following inequality holds:
\begin{equation}
\label{eqn:OldBound}
n_{\A}(C) \leq \max_{\sigma \in \mathfrak{S}_{n+2}} \, \sgnvar(b_{\sigma}).
\end{equation}
In this paper, we optimize this result as follows. Define the sequence
\[
\mu_\sigma = (\mu_0, \dots, \mu_{n+1}), \quad  \text{ with } \mu_j = \sum_{k=0}^j b_{\sigma_k}.
\]
These sums correspond to maximal minors of $A\ast \dots \ast A$.
We prove in Theorem~\ref{thm:refines} and Proposition~\ref{l:refines}
that for any uniform matrix $C$, the following sharper
inequality holds:
\begin{equation}
\label{eqn:NewBound}
n_{\A}(C) \leq 1+\max_{\sigma \in \mathfrak{S}_{n+2}} \, \sgnvar(\mu_{\sigma}).
\end{equation}
Actually, both results are stronger than~\eqref{eqn:OldBound} and~\eqref{eqn:NewBound} because
when the coefficient matrix $C$ satisfies the necessary
 condition~\eqref{eq:nonempty} for $n_\A(C)$ to be positive, we can associate to it a single
 permutation $\sigma$ without taking the maximum. Moreover, for general matrices
 $C$ of rank $n$ the sequences $b_\sigma$ and $\mu_\sigma$
 need to be replaced by further sums of their coefficients described in Definition~\ref{def:mu},  which depend on the pattern of minors of $C$
 that vanish. 

We show that our new bound always refines the bound in~\cite{BD} and that 
it is sharp in the following
sense: for any configuration $\A$ there exists a matrix $C$ such that $n_\A(C)$ equals the bound
we give in Theorem~\ref{thm:refines}.

We prove this in Theorem~\ref{thm:optimalmixed subdivision},
where we moreover relate the sums of minors for which the sign variation
is increased with volumes of  mixed cells {\em positively decorated by $C$} in an associated mixed subdivision.
Section~\ref{sec:2main} is centered around this second main result. We show that coefficient matrices for which the upper 
bound is attained might be obtained by Viro's method~\cite{Stu}.

\begin{example} \label{ex:intro}
Let $n=2$ and consider the three configurations:
\begin{align*}
\A_1 & = \{(0,0), (1,0), (1,2), (0,1)\}, \\
\A_2 & = \{ (0,0), (1,0), (1,1), (0,1)\}, \quad \text{and}\\
\A_3 &= \{(0,0), (3,0), (0,3), (1,1) \}.
\end{align*}
Figure~\ref{fig:MixedCells} depicts a mixed subdivision of each of the Minkowski sums $\A_i + \A_i$, $i=1,2,3$.
Notice that for the support set $\A_2$,  any mixed subdivision has
at most two mixed cells. In the other two cases, we show three mixed cells. But we will see in Example~\ref{ex:mixed3}
that if we consider $\A_3$, only two of the three mixed cells can be decorated by any matrix $C$. This is coherent with
our previous observation that $n_{\A_3}(C)$ cannot exceed $2$.
\begin{figure}[t]
\includegraphics[height=35mm]{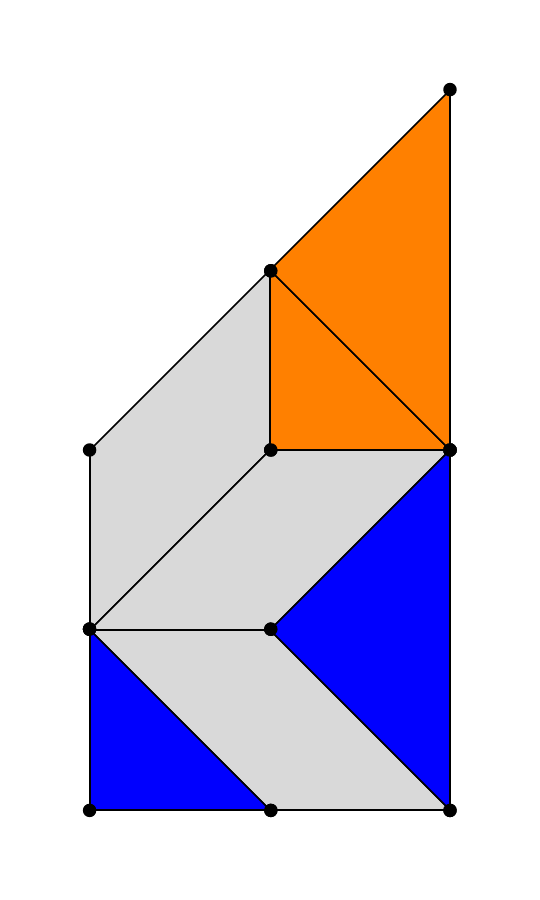}
\hskip10mm
\includegraphics[height=35mm]{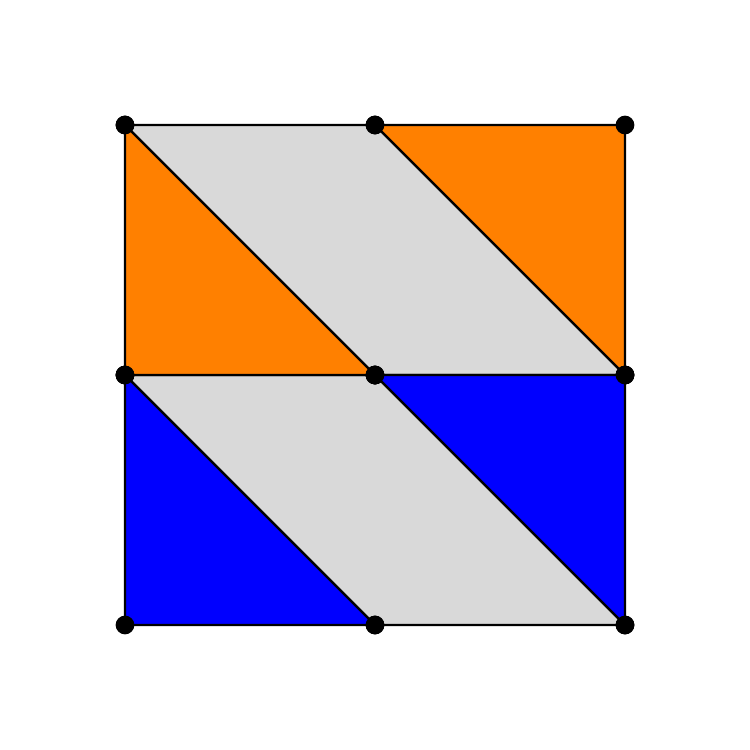}
\hskip10mm
\includegraphics[height=35mm]{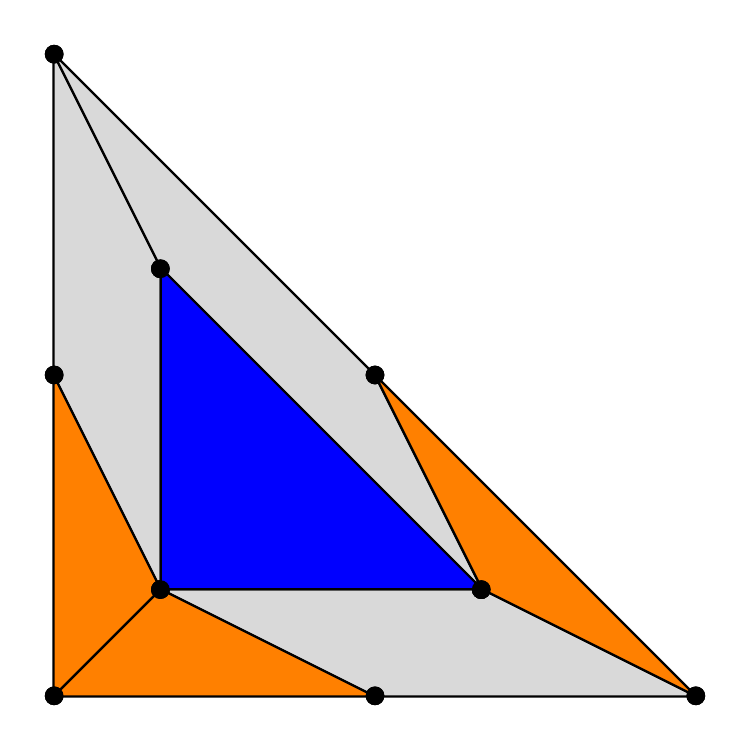}
\caption{The sign variations of $\mu_\sigma$ correspond to positively decorated mixed cells
of the mixed subdivisions.}
\label{fig:MixedCells}
\end{figure}
\end{example}

A lower bound  for $\max_{C} n_\A(C)$,
 also based on regular subdivisions of the Min\-kow\-ski sum $\A+ \dots + \A$ and Viro's method,  
was obtained in~\cite{ir96} for any value of $n$ and $m$.  These degeneration methods don't seem to be able to give an 

upper bound  for $n_\A(C)$ beyond the case $n=1$ or the case $m=n+2$ we solve here.
 
 We end the paper with the description of the moduli space of circuits in the plane which are the support of 
 bivariate polynomial systems  according to the maximal number of positive solutions. 
 
All our main results and proofs are also valid in case the configuration $\A$  of exponents  is real.

 \medskip

\noindent{\bf Acknowledgment:} We are grateful to the referee for several interesting comments; in particular, for pointing out that 
 inequality~\eqref{E:refines} could be interpreted as a discrete version of Rolle's theorem on $\R/k \Z$.

\bigskip

\section{Descartes' rule of signs for circuits} \label{sec:1main}

As we mentioned in the introduction, we will consider a configuration 
 $\A=\{a_0,a_1,\ldots,a_{n+1}\}$ in $\Z^n$ such that the convex
 hull of $A$ in $\R^n$ is $n$-dimensional.  In this section we present our optimal
 Descartes' rule of signs in Theorem~\ref{thm:refines} and we show in \S~\ref{ssec:BD} that this
 bound improves the previous bound in~\cite{BD}.

 We will restrict our study to the case
 in which $\A$ is a circuit, that is, when the configuration is minimally affinely dependent.
 In terms of the matrix $A$ in~\eqref{eq:A}  this condition
 means that all maximal minors of $A$  are nonzero. Equivalently, all entries of a Gale dual $B$ of $A$ are 
 nonzero by~\eqref{eq:ComplementaryMinors}.
 We refer the reader to Section~2 in~\cite{BD} for the easier case when this does not happen.

Given a coefficient matrix $C =(c_{i,j}) \in \R^{n \times (n+2)}$, there is a basic necessary condition 
for the existence of at least one positive solution  of the system \eqref{E:system}.  
Let $C_0, \dots, C_{n+1} \in \R^n$ denote the column vectors of the coefficient matrix $C$.
Given a solution $x \in \R_{>0}^n$,  we have that $(x^{a_0}, \dots, x^{a_{n+1}})$ is a positive vector in
the kernel of $C$ and so necessarily the origin lies in the positive cone generated by the columns of $C$:
\begin{equation}\label{eq:nonempty}
{\bf 0} \in \R_{>0} C_0 + \dots + \R_{>0} C_{n+1}.
\end{equation}
This condition depends on signs of minors of $C$.

\smallskip

As we mentioned, we will assume that the rank of $C$ is maximum.
A Gale dual of the matrix $C$ is given by a $(n+1)\times 2$ real matrix $P$ whose columns
are a basis of the kernel of $C$. In this case, we will denote by $P_0, \dots, P_{n+1}$ the
{\em row} vectors of $P$. The configuration $\{P_0, \dots, P_{n+1}\}$ is a Gale dual
configuration to the configuration $C_0, \dots, C_{n+1}$ of columns of $C$; it is unique up to linear tranformation.
The translation of the necessary condition~\eqref{eq:nonempty} to the Gale dual side
is the following. Any vector in the kernel of $C$ is of the form $P \cdot v$. Then, the existence
of a positive vector annihilated by $C$ is equivalent to the existence of a vector $v$ such that
\begin{equation}\label{eq:v}
\langle P_j, v \rangle >0, \text{  for all } j,
\end{equation} 
that is,  $P_0, \dots, P_{n+1}$ lie in an open halfspace.  In what follows we
will assume that this condition is satisfied (since otherwise $n_\A(C)=0$). 

Moreover, it is well known that maximal minors of $C$ coincide, up to a fixed nonzero constant, with
maximal minors of $P$ corresponding to complementary indices. Thus,  linear dependencies
on the Gale dual configuration reflect linear dependencies on the configuration
of columns of $C$:  two vectors $P_i, P_j$ are colinear if and only if the maximal 
minor of $C$ avoiding columns $i$ and $j$ is zero.
  
 We will use the notation $[s] = \{0,1, \dots, s-1\}$ for $s\in \N$.
 
 \begin{definition} \label{def:K}
 Let $C$ a maximal rank matrix satisfying~\eqref{eq:nonempty} and 
 $\{P_0, \dots, P_{n+1}\}$ a choice of Gale dual configuration.
We define an equivalence relation on the index set $[n+2]$ given by
\[
i \sim j \, \Leftrightarrow \, \det(P_{i}, P_{j}) = 0.
\]
We conclude that the set of equivalence classes (with cardinal $k \le n+1$)
\begin{equation}
\label{eq:Kj}
\sfrac{[n+2]}{\sim} \,= \, \{K_0, \dots, K_{k-1}\}
\end{equation}
has a canonical ordering (up to complete reversal).
By choosing an arbitrary
ordering within each equivalence class, we obtain a permutation
$\sigma \in \mathfrak{S}_{n+2}$
such that
\[
\epsilon \det(P_{\sigma_i}, P_{\sigma_j}) \geq 0, \quad i < j,
\]
where $\epsilon \in \{-1,1\}$ depends on the choice of orientation.
Let $K \subset [n+2]$ be a set of representatives of the equivalence classes
$K_0, \dots, K_{k-1}$.
Denote by $\bar{\sigma}$ the bijection
\[
\bar{\sigma}\colon [k] \rightarrow K
\]
which is induced by $\sigma$. Then,
$
\epsilon \det(P_{\bar{\sigma}_i}, P_{\bar{\sigma}_j}) > 0 \quad \text{if} \quad  i < j
$.
We say that $\sigma$ is \emph{an ordering for $C$} and that $\bar{\sigma}$ is \emph{a strict ordering} for $C$.
\end{definition}

Note that the previous definitions are {\em independent} of the choice of Gale dual configuration, since a Gale dual 
configuration is unique up to linear transformation. We present an example to clarify our definitions.

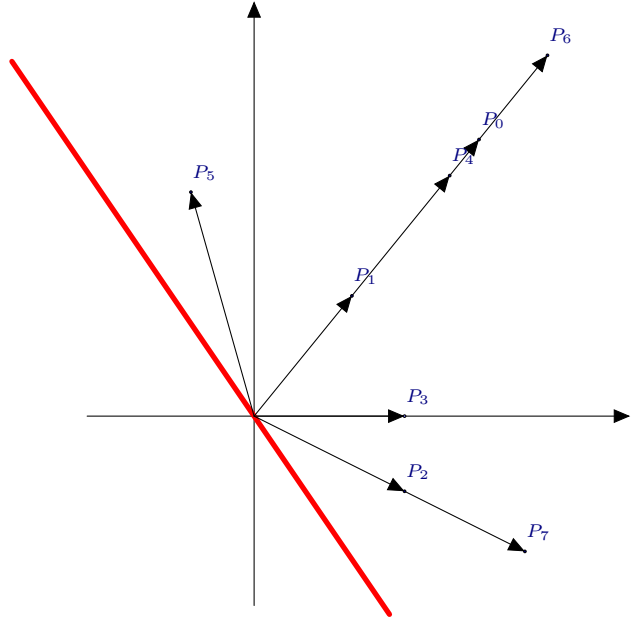
\begin{figure}[t]
\definecolor{dcrutc}{rgb}{0.8627450980392157,0.0784313725490196,0.23529411764705882}
\definecolor{xdxdff}{rgb}{0.49019607843137253,0.49019607843137253,1.}
\definecolor{qqqqff}{rgb}{0.,0.,0.5}
\definecolor{ffqqqq}{rgb}{1.,0.,0.}
\definecolor{uuuuuu}{rgb}{0.26666666666666666,0.26666666666666666,0.26666666666666666}
\begin{tikzpicture}[line cap=round,line join=round,>=triangle 45,x=1.0cm,y=1.0cm]
\draw[->,color=black] (-2.22,0.) -- (5.0,0.);
\draw[->,color=black] (0.,-2.52) -- (0.,5.52);
\draw [line width=2.pt,color=ffqqqq,domain=-3.22:1.80] plot(\x,{(-0.-5.1*\x)/3.48});
\draw [->] (0.,0.) -- (2.,-1.);
\draw [->] (2.,-1.) -- (3.6,-1.8);
\draw [->] (0.,0.) -- (2.,0.);
\draw [->] (0.,0.) -- (1.3,1.6);
\draw [->] (1.3,1.6) -- (2.6,3.2);
\draw [->] (2.6,3.2) -- (2.99,3.68);
\draw [->] (2.99,3.68) -- (3.9,4.8);
\draw [->] (0.,0.) -- (-0.84,2.98);
\begin{scriptsize}
\draw [fill=qqqqff] (2.,-1.) circle (0.5pt);
\draw[color=qqqqff] (2.18,-0.74) node {$P_2$};
\draw [fill=qqqqff] (3.6,-1.8) circle (0.5pt);
\draw[color=qqqqff] (3.78,-1.54) node {$P_7$};
\draw [fill=xdxdff] (2.,0.) circle (0.5pt);
\draw[color=qqqqff] (2.18,0.26) node {$P_3$};
\draw [fill=qqqqff] (1.3,1.6) circle (0.5pt);
\draw[color=qqqqff] (1.48,1.86) node {$P_1$};
\draw [fill=qqqqff] (2.6,3.2) circle (0.5pt);
\draw[color=qqqqff] (2.78,3.46) node {$P_4$};
\draw [fill=qqqqff] (2.99,3.68) circle (0.5pt);
\draw[color=qqqqff] (3.18,3.94) node {$P_0$};
\draw [fill=qqqqff] (3.9,4.8) circle (0.5pt);
\draw[color=qqqqff] (4.08,5.06) node {$P_6$};
\draw [fill=qqqqff] (-0.84,2.98) circle (0.5pt);
\draw[color=qqqqff] (-0.66,3.24) node {$P_5$};
\end{scriptsize}
\end{tikzpicture}
\caption{There cannot be more than $3$ positive roots}
\label{fig:gale}
\end{figure}

\begin{example}
Figure~\ref{fig:gale} features an example of a Gale dual configuration in case $n=6$ with a choice of ordering. 
Equivalence classes correspond to rays containing vectors $P_j$. Here, $k=4$. The associated
strict ordering $\bar{\sigma}\colon [4] \rightarrow K$ assigns $K_0 = \{ P_2, P_7\}$, $K_1=\{P_3\}$, $K_2=\{ P_1, P_4,
P_0, P_6\}$ and $K_3=\{P_5\}$. As we remarked,  this example does not correspond to a uniform matrix $C$.
We will see that in this case $n_\A(C)\le 3$ for any circuit $A$ in $\Z^6$. 
\end{example}

We need another definition in order to state our main result in this section.

\begin{definition} \label{def:mu}
We keep the assumptions and notations from Definition~\ref{def:K}.
For $\ell=0, \dots, k-1$, we set
\begin{equation}
\label{eqn:SetsL}
L_\ell= K_{\bar{\sigma}_{0}} \cup \dots \cup K_{\bar{\sigma}_{\ell}},
\end{equation}
and consider the two sequences $\lambda= \{\lambda_\ell\}_{\ell\in [k]}$ and $\mu = \{\mu_\ell\}_{\ell\in [k]}$ 
whose terms are defined by
\begin{equation}
\label{eqn:LambdaAndMu}
\lambda_\ell \, = \, \sum_{j \in K_{\bar{\sigma}_{\ell}}}
b_{j} 
\quad \text{and} \quad
\mu_{\ell} \, = \, \sum_{j \in L_{\ell}} b_{j}.
\end{equation}
\end{definition}

Notice that
\begin{equation}\label{eq:ls}
\mu_\ell=\lambda_0 + \dots + \lambda_\ell.
\end{equation}
In particular, $\mu_{k-1} = \sum_{\ell=0}^{k-1} \lambda_\ell = \sum_{j=0}^{n+1} b_j=0$, because the vector $(1, \dots, 1)$ is a row of $A$.
Also, both sequences depend on the strict ordering $\bar{\sigma}$, but we have not incorporated
this to the notation to make the displays look clearer. 

We are ready to state our first main result, in terms of this sequence $\mu$ constructed from  sums of minors of $A$ prescribed
by the pattern of signs of minors of $C$.

\begin{thm}\label{thm:refines} {\bf Descartes' rule of signs for circuits.}
Let $\A$ be a circuit in $\Z^n$. Let $C \in \R^{n \times(n+2)}$ a coefficient matrix of rank $n$ which satisfies~\eqref{eq:nonempty}
and $\sigma$ an ordering for $C$.  Consider the sequence $\mu$ from Definition~\ref{def:mu} associated to
the strict ordering $\bar{\sigma}$ induced by $\sigma$ in Definition~\ref{def:K}.

If $n_\A(C)$ is finite, then
\begin{equation}\label{Eq:mainDescartes}
n_\A(C)  \leq 1+\sgnvar(\mu).
\end{equation}
\end{thm}
In particular, $n_\A(C) \le k$. Also, $k \le n+1$ and equality holds if and only if $C$ is uniform. 

\begin{example}\label{ex:rectangle}
Consider  the configuration  of vertices of the unit square $\A_2 =\{ (0,0), (1,0), (1,1), (0,1)\}$.
 Let $C \in \R^{2 \times 4}$ be  a
uniform coefficient matrix such that  the identity is an ordering for $C$.  We can choose $B = (1, -1, 1, -1)^\top$
and so $\mu_0 = 1$, $\mu_1 = 1+(-1)=0$, $\mu_2= 1$, and $\mu_3=0$.  In this case, $n_{\A_2}(C) \le 
 1+\sgnvar(\mu) = 1$. Instead, if the associated ordering $\sigma$ for $C$ is given by $\sigma_{0}=0, \sigma_{1}=1$,
$ \sigma_{2}=3$, and $\sigma_{3}=2$, we get that $\mu_0 = 1$, $\mu_1 =0$, $\mu_2=-1$, and $\mu_3=0$. 
So, $n_{\A_2}(C) \le  1+\sgnvar(\mu) = 2$, 
 and this is the maximum possible upper bound given by Theorem~\ref{thm:refines}. 
 Indeed,  if $n_{\A_2}(C)$ is finite, it is at most ${\rm vol}_{\Z \A}(\A_2)=2$. 
 \end{example}

Before giving the proof of Theorem~\ref{thm:refines}, we first recall some basic notions from \cite[\S~4.1]{BD}.
A sequence $(h_1,h_2,\ldots,h_s)$ of real valued analytic functions  defined 
on an open interval $\Delta \subset \R$ satisfies Descartes' rule of signs on $\Delta$ if for any
sequence $c=(c_1,c_2\ldots,c_s)$ of real numbers, the number of roots 
 of $c_1h_1+c_2h_2+\cdots+c_sh_s$ in $\Delta$ counted with multiplicity never exceeds $\sgnvar(c)$.
The classical univariate Descartes' rule of signs asserts that monomial bases 
satisfy Descartes' rule of signs on the open interval $(0,+\infty)$.
Observe that $(h_1,h_2\ldots,h_s)$ satisfies Descartes' rule of signs 
on $\Delta$, if and only if $(h_s,h_{s-1},\ldots,h_1)$ has this property.

\smallskip

Recall that the Wronskian of $h_1,\ldots,h_s$ is the following determinant
$$
W(h_1,\ldots,h_s)= \mbox{det}
\left(
\begin{array}{cccc}
h_1 & h_2 & \ldots & h_s \\
h_1'& h_2' & \ldots & h_s'\\
\vdots & \vdots & \ldots & \vdots \\
h_1^{(s-1)} & h_2^{(s-1)} & \ldots & h_s^{(s-1)}
\end{array}
\right).$$
It is well known that $h_1,\ldots,h_s$ are linearly dependent if and only if their Wronskian
is identically zero.

\begin{prop}[{\cite[part 5, items 87 and 90]{P-S}}]\label{P:Descartes}
A sequence of functions $h_1,\ldots,h_s$ satisfies Descartes' rule of signs on $\Delta \subset \R$
if and only if for any collection of integers $1 \leq j_1 <j_2 \ldots < j_\ell \leq s$ we have
\begin{enumerate}
\item \label{E:cond1}
$
\quad W(h_{j_1},\ldots,h_{j_\ell}) \neq 0
$
\end{enumerate}
and for any collections of integers $1 \leq j_1 <j_2 \ldots < j_\ell \leq k$ 
and $1 \leq j_1' <j_2' \ldots < j_\ell' \leq k$ of the same size, we have
\begin{enumerate}
\setcounter{enumi}{1}
\item \label{E:cond2}
$\quad W(h_{j_1},\ldots,h_{j_\ell}) \cdot W(h_{j_1'},\ldots,h_{j_\ell'}) >0.\hfill\qed$
\end{enumerate}
\end{prop}
If a sequence of analytic functions $h_1,\ldots,h_s$ satisfies Descartes' rule of signs on $\Delta \subset \R$, then $h_1, \dots, h_s$
do not vanish, have the same sign on $\Gamma$ and no two of them are proportional. 

We will also need the following proposition, which is inspired in Lemma~2.1 (Fekete's Lemma) in~\cite{Pinkus} for the case of totally positive matrices.

\begin{prop}\label{prop:Fekete}
Let  $V$ be a $t \times k$ real matrix, $t \le k$, with the following property: for any $s \le t$ there exists a sign $\sigma_s \in \{ -1,1\}$ 
such that for any choice of $s$ {\sl consecutive} indices $j_1, j_1+1, \dots, j_1 +s-1$, the sign of the determinant of the submatrix of $V$ consisting of the first $s$ 
rows and the $s$ consecutive columns with these indices, equals $\sigma_s$. Then, the sign of any $s \times s$ minor of $V$ consisting of the first $s$ 
rows and any $s$ columns equals $\sigma_s$. 
\end{prop}

\begin{proof}
We need to prove that all  $s \times s$ minors of $V$ consisting of the first $s$ 
rows and any $s$ columns have the same sign $\sigma_s$ shared by all $s\times s$ minors involving the first $s$  rows and $s$ consecutive columns.
For any $1 \le s \le t$ and any subset  $J = \{j_1, \dots, j_{s}\}$ with $1 \leq j_{1}< j_{2}<\cdots <j_{s} \leq k$,  denote by $V_J$ the submatrix of $V$ 
consisting of the first $s$ rows and the (ordered) columns in $J$. Define the associated dispersion number $d(J)$ to be the number of integer points in 
the interval $[j_{1},j_{s}]$ which are distinct from $j_1,\ldots,j_{s}$.  Note that $d( j_{1}, j_{2},\cdots ,j_{s})=0$ means that $ j_{1}, j_{2},\cdots ,j_{s}$ are consecutive integer numbers.

We prove the result by double induction on $s \geq 1$ and $d \geq 0$. The result is obvious when $s=1$ (which implies $d=0$)
and we are assuming that the result holds for any $s$ when $d=0$. 

Let $s \geq 2$ and $d \geq 1$ (and so $s < k$). Assume by inductive hypothesis that the result holds for all  $s' \le  s$ and all $d' \le d$ such that at least one of the inequalities is strict. 
Let $J =\{ j_1, \dots, j_s\}$, $1 \leq j_{1} < \cdots <j_{s} \leq k$, with $d(J)=d$. Then, there exists another index $j$ such that $j_1 < j < j_s$. Note that $d(J \cup \{j\})= d(J)-1$. 

Now,  by Equality (1.2) in Chapter~1 of ~\cite{Pinkus}, we have
\begin{equation} \label{12}
\begin{array}{c}
\det(V_J) \cdot \det(V_{J \setminus \{j_1, j_s\} \cup \{j\} })
=
\\

\det(V_{(J \setminus \{j_1\}) \cup \{j\} }) \cdot \det(V_{J \setminus \{j_s\} })

+
\det(V_{(J \setminus \{j_s\}) \cup \{j\} }) \cdot \det(V_{J \setminus \{j_1\} }).
\end{array}
\end{equation}
Note that the second factor in each summand in~\eqref{12} corresponds to a minor of size $s-1$ and so, all of them have the same sign $\sigma_{s-1}$. Also
\[d((J \setminus \{j_1\}) \cup \{j\})\le d (J \cup \{j\}) < d(J),\] 
and similarly  $d((J \setminus \{j_s\}) \cup \{j\})\le d (J \cup \{j\}) < d(J)$.
Again by inductive hypothesis, the determinants of $V_{(J \setminus \{j_1\}) \cup \{j\}}$ and $V_{(J \setminus \{j_s\}) \cup \{j\}}$ have the same sign $\sigma_s$.
It follows that the sign of $\det(V_J)$ also equals $\sigma_s$, as wanted.
\end{proof}

Consider again a coefficient matrix $C \in \R^{n \times (n+2)}$ of rank $n$ satisfying~\eqref{eq:nonempty}.
Since a Gale dual  configuration is defined up to a linear transformation, 
we will assume without loss of generality that we have a Gale dual configuration $\{P_0, \dots, P_{n+1}\}$  
such that the  vector $v$ in~\eqref{eq:v} has first coordinate equal to $1$.
We define the associated linear functions 
\begin{equation}\label{eq:p}
p_{j}(y) = \langle P_{j}, (1,y) \rangle,
\end{equation}
and the open segment
\[
\Delta_P=\{y \in \R \, : \, p_{j}(y)>0\}.
\]
Note that $\Delta_P$ is nonempty because of our hypothesis about $v$, and $p_{0}\dots, p_{n+1}$ are positive linear functions on $\Delta_P$.
The following basic proposition is crucial.

\begin{prop}[{\cite{BS07}, see also \cite[\S 4]{BD}}]
\label{prop:Galebijection}
Consider the polynomials $p_{j}$ in~\eqref{eq:p} and define
the following function $g : \Delta_P \to \R$  
\begin{equation}\label{eq:g}
g(y) = \prod_{j \in [n+2]}  p_j(y)^{{b}_j}.
\end{equation}
Then, $n_\A(C)$ equals the number of solutions, counted with multiplicities, of the equation \begin{equation}
\label{eqn:GaleDualSystem}
g(y)=1
\end{equation}
in the interval $\Delta_P$. 
\end{prop}

Proposition~\ref{prop:Galebijection} follows from the following two facts: i) all vectors in the kernel of $C$ are of the form
$(p_0(y), \dots, p_{n+1}(y))$ for $y \in \R^2$, and ii) assuming all $p_i(y)$ are positive, 
there exists a positive vector $x$ with
nonzero coordinates such that $x^{a_j}= p_j(y)$ for all $j$ if and only if $g(y)=1$. 

\medskip

\noindent{\bf Proof of Theorem~\ref{thm:refines}.}
Let $\A$, $C$ and the ordering $\sigma$ be as in the statement of Theorem~\ref{thm:refines}. 
Choose a Gale dual configuration $\{P_0, \dots, P_{n+1}\}$ of the columns of $C$ satisfying~\eqref{eq:v}.
Consider the linear forms $p_{j}$ as in \eqref{eq:p} and the rational function $g\colon \Delta_P \to \R$ defined in~\eqref{eq:g}.
By Proposition~\ref{prop:Galebijection}, it is enough to bound the number of solutions, counted with multiplicities, of the equation \eqref{eqn:GaleDualSystem} in 
$\Delta_P=\{y \in \R \, | \, p_{j}(y)>0\}$. 
As we assume that $n_\A(C)$ is finite, we get that $g(y)\not \equiv 1$.

Recall the partition of $[n+2]$ defined in~\eqref{eq:Kj}. For any $\ell \in [k]$, an index $j \in K_\ell$ if and only 
if there exists a positive constant $e_{j,\ell}$ such that $P_j = e_{j, \ell} \,P_{{\bar{\sigma}_\ell}}$.
Thus, there is a positive constant $e$ such that
\begin{equation}\label{eq:gbis}
g(y) = e \cdot \prod_{\ell \in [k]} p_{\bar{\sigma}_{\ell}}(y)^{\lambda_\ell},
\end{equation}
Consider the associated sequence $\mu$ as in Definition~\ref{def:mu}.
We may rewrite $g(y)$ as
\begin{equation}\label{eq:gbis}
g(y) = e \cdot \prod_{\ell \in [k-1]} \left( \frac{p_{\bar{\sigma}_\ell}(y)}
{p_{\bar{\sigma}_{\ell+1}}(y)}\right)^{\mu_\ell}.
\end{equation}
The logarithm $G(y) = \log(g(y))$ is well defined over $\Delta_P$ and for any $y \in \Delta_P$
we have $g(y) =1 $ if and only if $G(y) =0$.
The derivative of $G$ over $\Delta_P$ equals
\begin{equation}\label{eq:lambdabar}
G'(y) = \sum_{{\ell \in [k-1]}} \mu_\ell
\left(
\frac{p'_{\bar{\sigma}_\ell}}{p_{\bar{\sigma}_\ell}} - 
\frac{p'_{\bar{\sigma}_{\ell+1}}}{p_{\bar{\sigma}_{\ell+1}}}
\right).
\end{equation}
Observe  that  there exists $\epsilon \in \{+1,-1\}$ such that, for $j > i$,
\begin{equation}
\label{eqn:DifferenceIsDeterminant}
\epsilon \cdot \left(\frac{p'_{\bar{\sigma}_i}}{p_{\bar{\sigma}_i}}-\frac{p'_{\bar{\sigma}_j}}{p_{\bar{\sigma}_j}} \right) =
\epsilon \cdot
\frac
{\det
(P_{\bar{\sigma}_j},P_{\bar{\sigma}_i})}
{p_{\bar{\sigma}_i}\, p_{\bar{\sigma}_j}} > 0  \quad \text{\, on \, }  \Delta_P.
\end{equation}

We moreover claim that  the  functions $\{\frac{p'_{\bar{\sigma}_\ell}}{p_{\bar{\sigma}_\ell}} - 
\frac{p'_{\bar{\sigma}_{\ell+1}}}{p_{\bar{\sigma}_{\ell+1}}}, \ell 
\in [k-1]\}$ satisfy the hypotheses of Proposition~\ref{P:Descartes}.
By Proposition~\ref{prop:Fekete} it is enough to consider the case in which the indices are consecutive. Moreover, for simplicity, it is enough to
compute the Wronskian
\[
W\bigg(\frac{p'_1}{p_1}-\frac{p'_2}{p_2},\, \frac{p'_2}{p_2}-\frac{p'_3}{p_3},\ldots, \frac{p'_{\ell}}{p_{\ell}}-\frac{p'_{\ell+1}}{p_{\ell+1}}\bigg).
\]
We have
$$
\bigg(\frac{p'_j}{p_j}\bigg)^{(k-1)} =(-1)^{k-1}\, (k-1)!  \bigg(\frac{p'_j}{p_j}\bigg)^{k},
$$
and a computation of a Vandermonde determinant leads to
\begin{align*}
W\bigg(\frac{p'_1}{p_1}-\frac{p'_2}{p_2},\,\ldots, \frac{p'_{\ell}}{p_{\ell}}-\frac{p'_{\ell+1}}{p_{\ell+1}}\bigg) & =
(-1)^{\ell}\gamma_\ell \cdot \mbox{det}
\left( \Big(\frac{p'_j}{p_j}\Big)^{k-1}
\right)_{1\leq k,j \leq \ell+1} \\
& = \displaystyle(-1)^{\ell} \gamma_\ell \cdot \prod_{1 \leq j < i \leq \ell+1} \Big(\frac{p'_i}{p_i}-\frac{p'_j}{p_j}\Big),
\end{align*}
where $\gamma_\ell$ denotes the product
$\gamma_\ell \, = \, \prod_{k=1}^{\ell}  (-1)^{k-1} (k-1)! $.  The above Wronskian is non-vanishing, as the determinants 
$\det(P_j, P_i)$ are non-vanishing, and its sign only depends on $\ell$.

Then, our collection satisfies conditions \eqref{E:cond1}
and \eqref{E:cond2} in Proposition \ref{P:Descartes} and so
the number of roots of
$G'$ on $\Delta_P$ counted with multiplicities is at most $\sgnvar(\mu)$.
Rolle's theorem leads to the bound $1+\sgnvar(\mu)$ for the number of roots of $g$ contained in $\Delta_P$ counted with multiplicities, 
and thus to this bound for the number $n_\A(C)$, as wanted.

\subsection{Our bound refines Theorem~2.9 in \cite{BD}}\label{ssec:BD}
Let ${\rm vol}_{\Z}(\A)$ denote the  normalized volume of the configuration $\A \subset \Z^n$ with respect to  the lattice $\Z^n$,  
which is the Euclidean volume of the convex hull of $\A$ multiplied by $n!$.  By Bernstein-Kouchnirenko Theorem, this volume ${\rm vol}_{\Z}(\A)$ bounds the number of isolated solutions of
system~\eqref{E:system} in the complex torus. 

\begin{remark} \label{rem:vol}
Dividing ${\rm vol}_{\Z}(\A)$ by the index $I$ of the affine sublattice $\Z{\A}$ generated by $\A$ gives the normalized volume ${\rm vol}_{\Z\A}(\A)$ with respect to the lattice $\Z\A$. 
Then ${\rm vol}_{\Z\A}(\A)$ bounds the number of isolated positive solutions of system~\eqref{E:system}. 
This index $I$ is equal to the greatest common divisor of the maximal minors of the matrix  $A$ and it is also equal to the
 greatest common divisor of the coefficients $b_j = (-1)^j \det(A(j))$ which give a choice of $B$, as in~\eqref{eq:ComplementaryMinors}. 
Since $\A$ is a circuit, Proposition 1.2 of Chapter 7 in \cite{GKZ}  shows that
the convex hulls of the subsets ${\rm conv}(\A \setminus \{a_j\})$ for all $j$ such that $b_j>0$ give a triangulation of
 ${\rm conv}(A)$. Thus, the Euclidean volume of ${\rm conv}(\A \setminus \{a_j\})$ equals $\frac {b_j}{n!}$ and so 
 $\vol_\Z({\rm conv}(\A \setminus \{a_j\}) = b_j$. Therefore, we have that $\vol_{\Z}(\A)=\sum_{b_j >0} b_{j}$,
from which  $\vol_{\Z\A}(\A)=\frac{1}{I}\sum_{b_j >0} b_{j}$ (cf. \cite[Section 2.1]{BD}). 
\end{remark}

With the same hypotheses about $\A,C$, and $\sigma$ as in Theorem~\ref{thm:refines}, let $\lambda$ be the sequence in Definition~\ref{def:mu}.  
Theorem 2.9 in~\cite{BD} proves that if $n_\A(C)$ is finite, then
\[
n_\A(C) \leq {\rm max}(\sgnvar(\lambda), {\rm vol}_{\Z\A}(\A)).
\]
We show in Proposition~\ref{l:refines} and Lemma~\ref{prop:VolumeBound} below that our bound always refines this result. We first give an example where
we get a strict inequality.

\begin{example}[Example~\ref{ex:rectangle}, continued] \label{ex:suite}
Consider again  the configuration  $\A_2 =\{ (0,0), (1,0), (1,1), (0,1)\}$
and choose $B = (1, -1, 1, -1)^\top$.  Let $C \in \R^{2 \times 4}$ be a
uniform coefficient matrix such that  the identity is an ordering for $C$. 
Then, $\sgnvar(\lambda)=3$ but $1+\sgnvar(\mu) = 1$. Moreover, we saw
that $n_{\A_2}(C) \le 2$ for any $C$. 
 \end{example}
 
 We will use the following definition.
 
 \begin{definition}\label{indicessignchanges}
Let $\lambda=(\lambda_0,\lambda_1,\ldots, \lambda_{k-1})$ be a finite 
a sequence of real numbers with $\sgnvar(\lambda)=m$. The \emph{indices of sign changes of $\lambda$} are 
the indices $\{\ell_0, \dots, \ell_m\} \subset [k]$ defined recursively by
\begin{enumerate}
\item $\ell_0 = \min \big\{\ell \in \{0,\dots, k-1\} \,\big| \, \lambda_\ell \neq 0 \big\}$, and\vskip1mm
\item $\ell_i = \min \big\{\ell \in \{\ell_{i-1}, \dots, k-1\} \, \big| \, \lambda_\ell \, \lambda_{\ell_{i-1}} < 0\big\}$.
\end{enumerate}
\end{definition}

Part of the following result is stated as Problem~7.4.3 in~\cite{Ph}. 

 \begin{prop}\label{l:refines}
Let $\lambda=(\lambda_0, \lambda_1,\ldots,\lambda_{k-1})$ be a nonzero sequence of real numbers such that $\sum_{\ell \in [k]} \lambda_\ell=0$, 
and set $\mu_\ell= \lambda_0 + \dots + \lambda_\ell$ for $\ell \in [k]$. Then,
\begin{equation}\label{E:refines}
1+\sgnvar(\mu) \leq \sgnvar(\lambda).
\end{equation}
and
\begin{equation}\label{E:mod2}
\sgnvar(\lambda) \equiv 1+\sgnvar(\mu) \mod 2.
\end{equation}
In addition, if $\lambda$ is an integer sequence, with $V=\sum_{\lambda_\ell >0} \lambda_\ell$, then
\begin{equation}\label{E:lessthanvolume}
1+\sgnvar(\mu) \leq V.
\end{equation}
\end{prop} 
\begin{proof}
Let $m=\sgnvar(\mu)$, and let
$\ell_0<\ell_1<\cdots<\ell_m$ be the indices of sign changes in the sequence
$\mu$. Assume without loss of generality that $\mu_{\ell_0} >0$ (so that the first
non-zero entries of both $\mu$ and $\lambda$ are positive). Then,
\[
\sgn(\mu_{\ell_s})=(-1)^{s}, \quad s=0,\ldots,m.
\]
In addition,
\[
\sgn(\lambda_{\ell_s})=\sgn(\mu_{\ell_{s}}-\mu_{\ell_{s}-1}) = (-1)^{s},
\]
since either $\mu_{\ell_{s}-1} = 0$ or $\sgn(\mu_{\ell_{s}-1})=-\sgn(\mu_{\ell_{s}})$.
It follows that
\[
r \vcentcolon= \sgnvar(\lambda) - m \geq 0.
\]
Hence, to prove \eqref{E:refines} and \eqref{E:mod2},
it remains only to prove that $r$ is odd.
There is no loss of generality in assuming that $\mu_{k-2} \neq 0$.
Since $\mu_{k-1}=\mu_{k-2}+\lambda_{k-1}=0$,
we conclude that $\lambda_{k-1} \neq 0$ and
\[
\mu_{k-2} \, \lambda_{k-1} <0.
\]
Then, $\mu_{k-2}$ and $\lambda_{k-1}$ are the last non-zero entries
of the sequences $\mu$ respectively $\lambda$.
In particular,
\[
\sgn(\mu_{k-2}) = (-1)^m
\quad \text{ and } \quad
\sgn(\lambda_{k-1}) = (-1)^{m+r}.
\]
It follows that
\[
-1 = \sgn(\mu_{k-2} \, \lambda_{k-1}) = (-1)^{2m+r},
\]
which proves that $r$ is odd.

Assume in addition that $\lambda$ is an integer sequence.
Then, $\lambda_{\ell_{0}} \geq 1$. Moreover, for $s=1,\ldots,m$ we get
\[
2 \leq (-1)^s\big(\mu_{\ell_s} - \mu_{\ell_{s-1}}\big)
\]
since $\mu_{\ell_{s}}$ and $\mu_{\ell_{s-1}}$ are nonzero integers with opposite signs.
In particular, if $s = 2t$ is even, then
\[
2 \leq \sum_{j=\ell_{s-1}+1}^{\ell_s} \lambda_j \leq \sum_{ j = \ell_{s-1} +1}^{\ell_s} \max(\lambda_j, 0).
\]
It follows that,
\[
V=\sum_{j=0}^{k-1}  \max(\lambda_j, 0)
 \geq \lambda_{\ell_0}+\sum_{t=1}^{\lfloor \frac{m}{2} \rfloor} 
 \sum_{j=\ell_{2t-1} +1}^{\ell_{2t}} \max(\lambda_j,0)
 \geq 1 + 2\left\lfloor \frac{m}{2} \right \rfloor,
\]
which gives the desired inequality when $m$ is even. When $m$ is odd, we proceed similarly,
using instead that if $s = 2t-1$, then
\[
2 \leq \sum_{j=\ell_{s-1}+1}^{\ell_s}-\lambda_j \leq \sum_{ j = \ell_{s-1} +1}^{\ell_s} -\min(\lambda_j, 0).\qedhere
\]
\end{proof}

We now give a direct proof that our bound in Theorem~\ref{thm:refines} refines the bound $\vol_{\Z\A}(\A)$  from Bernstein-Kouchnirenko Theorem for integer circuits $\A$.

\begin{lemma}
\label{prop:VolumeBound}
With the same hypotheses about $\A,C$, and $\sigma$ as in Theorem~\ref{thm:refines}, take
$b_j = (-1)^j \, \det(A(j))$ as in~\eqref{eq:ComplementaryMinors} and
let $\mu$ be the sequence in Definition~\ref{def:mu}. 
We have the inequality
\begin{equation}\label{Eq:mainDescarteslessvolume}
1+\sgnvar(\mu) \leq \vol_{\Z\A}(\A).
\end{equation}
\end{lemma}

\begin{proof}
Let $B = (b_0,\dots, b_{n+1})^\bot$ be a Gale dual of $A$ and $I = \gcd(b)$. Recall from Remark~\ref{rem:vol} that
$\vol_{\Z\A}(\A)=\frac{1}{I}\sum_{b_j >0} b_{j}$. 
The inequality 
\[
1+\sgnvar(\mu) \leq \vol_{\Z\A}(\A)
\]
follows from \eqref{E:lessthanvolume} applied to the sequence
$s=(\frac{b_0}{I},\ldots,\frac{b_k}{I})$.
\end{proof}

Recall that $\A$ is
said to have a Cayley structure if there is a nontrivial partition of the indices
$[n+2]= J \cup J^c$ such that $\sum_{j \in J} b_j = \sum_{j \notin J} b_j =0$,
where $b_0, \dots, b_{n+1}$ are the coefficients of a Gale dual matrix $B$ of $A$.  
We obtain a result analogous to \cite[Proposition 2.12]{BD}. 

\begin{prop}\label{prop:congruence}
Let $\A$ and $C$ be as in the statement of Theorem~\ref{thm:refines}, and let
$\lambda_0, \dots, \lambda_{k-1}$ corresponding to a strict ordering for $C$ as
in Definition~\ref{def:mu}. Assume moreover that
$\lambda_0$ and $\lambda_{k-1}$ are different from $0$. Then, 
\begin{equation}
\label{eqn:DifferenceIsEven}
1+\sgnvar(\mu)\equiv n_\A(C) \mod 2.
\end{equation}
Therefore, $n_\A(C) > 0$ if $\sgnvar(\mu)$ is even.

In particular, \eqref{eqn:DifferenceIsEven} holds for any $C$ if $\A$ does not have a Cayley structure, 
and it holds for any $\A$ if $C$ is uniform.
\end{prop}

\begin{proof}
Assume that $\lambda_0, \lambda_{k-1} \neq 0$. 
Notice that $\mu_0=\lambda_0 \neq 0$.  We also have that
$0=\mu_{k-1}= \mu_{k-2}+\lambda_{k-1}$, which gives $\mu_{k-2}=-\lambda_{k-1} \neq 0$.
Hence,
\[
\sgnvar(\mu)=\sgnvar(\mu_0,\mu_1,\ldots,\mu_{k-2}) \equiv \sgnvar(\mu_0,\mu_{k-2}) \mod 2.
\] 
We also get
 \[
 \sgnvar(\mu_0,\mu_{k-2}) + \sgnvar(\lambda_0, \lambda_{k-1})=1.
 \] 
 The result follows now from \cite[Proposition 2.12]{BD} and the fact that $\sgnvar(\lambda_0, \lambda_{k-1}) \equiv \sgnvar(\lambda_0, \ldots, \lambda_{k-1}) \, \mod 2$.
\end{proof}

\section{Optimality and interpretation of our bound} \label{sec:2main}

In this section we prove our second main result Theorem~\ref{thm:optimalmixed subdivision}, where
we prove that our upper bound in Theorem~\ref{thm:refines} is sharp and we give an interpretation
of the coefficients in the sequence $\mu$ associated to a matrix $C$ and strict ordering $\sigma$ in terms
of decorated mixed cells in a mixed subdivision of the Minkowski sum of the supports of the
polynomial system~\ref{E:system}.
We give a very quick introduction to these notions in \S~\ref{ssec:basic}, and we refer the reader to~\cite{BDG,GKZ,HRS,Stu}
for further details. 

\subsection{Some previous notions and results}\label{ssec:basic}

Let $\A \subset \Z^n$ be a finite configuration.
A {\it real Viro polynomial system} is a polynomial system with exponents in $\A$ and 
coefficient matrix 
\begin{equation}\label{eq:Ct}
C_{t}=(c_{i,j}\,t^{h_{i,j}}),
\end{equation}
where $t$ is a real parameter and both the coefficient matrix $C=(c_{i,j})$ and  the lifting matrix $H=(h_{i,j}) $ are real matrices of the same size.
 Let $\A_{i}$ denote the support of the $i$-th equation, that is, $\A_{i}=\{a_{j} \in \A \, , \, c_{ij} \neq 0\}$.
 
The matrix $H$ defines a regular subdivision of the associated Cayley configuration
 $\A_{1} \ast \cdots \ast \A_{n}$ as well as a regular  mixed subdivision of the Minkowski sum
 $\A_{1}+\cdots+\A_{n}$.
These two subdivisions determine each other via the combinatorial Cayley trick (see \cite{Stu-Cayley} or \cite{HRS}).
That is, each cell of the regular subdivision of $\A_{1} \ast \cdots \ast \A_{n}$ determined by $H$ can be written as a Cayley configuration
$\sigma_{1} \ast \cdots \ast \sigma_{n}$, where the Minkowski sum $\sigma_{1}+\cdots+\sigma_{n}$ is a cell of the regular mixed 
subdivision of $\A_{1}+\cdots+\A_{n}$ determined by $H$.
Conversely, each cell of the mixed subdivision of $\A_1 + \dots + \A_n$ arises this way. 

Consider a regular triangulation of the Cayley configuration $\A_{1} \ast \cdots \ast \A_{n}$. 
A cell $\sigma_{1}+\cdots+\sigma_{n}$ of the associated mixed subdivision of $\A_1 + \cdots + \A_n$
is called {\it mixed} if $\sigma_{1},\ldots,\sigma_{n}$ are segments (consisting of two points), and it is called {\it positively decorated} by $C$ when for $i=1,\ldots,n$,
 if $a_{i_{1}}$ and $a_{i_{2}}$ are the vertices of $\sigma_{i}$, then $c_{i, i_{1}} \cdot  c_{i, i_{2}} <0$. Equivalently,
 a mixed cell $\sigma_{1}+\cdots+\sigma_{n}$ is positively decorated by $C$ if the binomial system $c_{i, i_{1}} x^{a_{i_{1}}}+c_{i, i_{2}} x^{a_{i_{2}}}=0$
 has exactly one positive solution. 
 
\begin{example}\label{ex:mixed3}
Let $\A_3= \{ (0,0), (3,0), (0,3), (1,1)\}$ as in Example~\ref{ex:intro} in the Introduction.
We reproduce in Figure~\ref{fig:MixedCells3} the rightmost picture in Figure~\ref{fig:MixedCells}.
\begin{figure}[t]
\includegraphics[height=35mm]{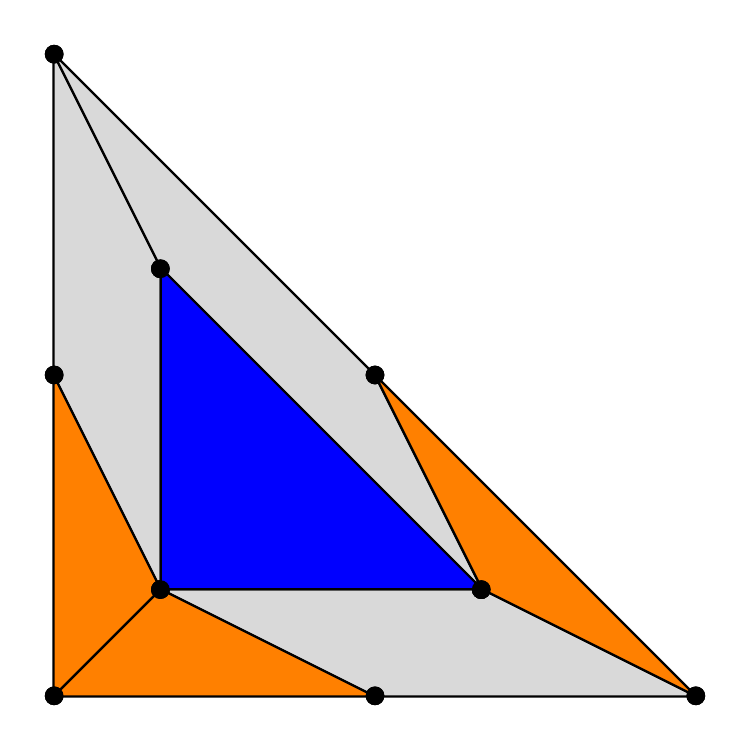}
\caption{A regular
mixed subdivision of $\A_3+\A_3$.}
\label{fig:MixedCells3}
\end{figure}

The vertices of $\A_3+ \A_3$ are the points $(0,0)=(0,0) + (0,0)$,
$(6,0) = (3,0) + (3,0)$ and $(0,6)=(0,3)+(0,3)$. The inner blue triangle 
is the cell $(1,1)+\A$  (or $\A+(1,1)$). 
This imposes that the three grey mixed cells obtained by projecting the lower hull of the lifting
induced by a matrix $H$ have to be sums of two segments,  for any  $H$ defining the blue cell.
The three mixed cells in grey are then necessarily as follows:
\begin{itemize}
\item The mixed cell at the bottom equals the Minkowski sum 

$\{(1,1), (3,0)\}+\{(0,0), (3,0)\}$.
\item The mixed cell at the top on the left equals $\{(1,1),(0,3)\}  +\{(0,0), (0,3)\}$.
\item The mixed cell at the top on the right equals   $\{(1,1),(0,3)\} + \{(3,0), (0,3)\}$.
\end{itemize}
These three mixed cells cannot be simultaneously positively decorated:
looking at the second summand, this would give opposite signs to $(0,0)$ and (3,0),
then opposite signs to $(0,0)$ and $(0,3)$, and also opposite signs to $(3,0)$ and $(0,3)$, which is impossible.
\end{example}

 The following result is well known, it is a particular case of Theorem~4 in~\cite{Stu}.

 \begin{lemma}\label{Stu}
 With the previous notations, there exists $t_{0}>0$ such that for any real number $t$ satisfying $0<t<t_{0}$ the number of positive solutions of the polynomial system
 $$\sum_{j=1}^{s} c_{ij}t^{h_{ij}}x^{a_{j}}=0, \; i=1,\ldots,n,$$
 is equal to the number of positively decorated mixed polytopes of the associated mixed subdivision. \qed
 \end{lemma}
                                                             
We will show that the bound \eqref{Eq:mainDescartes} in Theorem~\ref{thm:refines}
can be achieved by a Viro polynomial system.

\subsection{Statement and proof of Theorem~\ref{thm:optimalmixed subdivision}}

With the notations of Theorem~\ref{thm:refines}, it is implied by Theorem 1.1 in~\cite{H} that given a circuit $\A$,
there exists a matrix $C$ such that $n_\A(C)=n+1$ if and only if $k= n+2$, and 
there exists a bijection $\sigma:[n+2] \rightarrow [n+2]$ such that $\sgnvar(\mu)=n$.
This result can be recovered as a special case of Theorem~\ref{thm:refines} and the following statement that we
prove in our main result Theorem \ref{thm:optimalmixed subdivision}.

Let $\A$ be a circuit in $\Z^n$ and let $\{K_0, \dots, K_{k-1}\}$ be a partition of $[n+2]$. Consider any bijection 
$\bar{\sigma}: [k] \rightarrow K$ , where $K$ is any set of representatives of $K_0, \dots, K_{k-1}$.
Let $\mu$ be the sequence defined by \eqref{eqn:LambdaAndMu}.
Then, there exists a matrix $C \in \R^{n \times (n+2)}$ of maximal rank satisfying ~\eqref{eq:nonempty} such that  $\bar{\sigma}$ is a strict ordering of
$C$ and
\begin{equation}\label{E:optimal}
n_\A(C)=1+\sgnvar(\mu)
\end{equation}
We will moreover show that the numbers $\mu_{\ell}$ from \eqref{eq:ls}
for those indices $\ell$ which gives a sign change in the sequence \eqref{Eq:mainDescartes} are in absolute value equal to the Euclidean volumes 
of positively decorated mixed cells in a mixed subdivision. We first prove the following Lemma.

\begin{lemma}
\label{lem:MinkowskiSumIsLambdaSum}
Let $\A= \{ a_0, \dots, a_{n+1}\} \subset \Z^n$ be a circuit and 
set  $b_j = (-1)^j \, \det(A(j))$, $j=0,\ldots,n+1$. For each $\ell \in [n+1]$, the Euclidean volume of the Minkowski sum
\begin{equation}
\label{eqn:ConvexHullSegments}
\sum_{j=1}^{\ell} \{a_{0},a_{j}\}+\sum_{j =\ell+1}^{n} \{a_{n+1},a_{j}\}
\end{equation}
is equal to the absolute value of $\sum_{j=0}^{\ell} b_{j}$.
\end{lemma}

\begin{proof}
There is no loss of generality in assuming that $a_{n+1} = 0$. Let $V$ denote the Euclidean volume of the Minkowski sum \eqref{eqn:ConvexHullSegments}. Then,
\[
n! \, V = \pm \det(a_1-a_0, \dots, a_\ell - a_0, a_{\ell+1}, \dots, a_n).
\]
Let $D = \det(a_1-a_0, \dots, a_\ell - a_0, a_{\ell+1}, \dots, a_n)$.
By expansion, we find that
\begin{align*}
D & =  \det(a_1, \dots, a_\ell, a_{\ell+1}, \dots, a_n) \\
& \qquad - \sum_{j=1}^\ell \det(a_1, \dots, a_{j-1}, a_0, a_{j+1}, \dots  a_\ell, a_{\ell+1}, \dots, a_n) \\
& = \sum_{j=0}^\ell (-1)^{j}\det(a_0, \dots, a_{j-1}, a_{\hat \jmath}, a_{j+1}, \dots, a_n).
\end{align*}
Since $a_{n+1} = 0$, we have that
\[
(-1)^j\det(a_0, \dots, a_{j-1}, a_{\hat \jmath}, a_{j+1}, \dots, a_n) = (-1)^j \det A(j) = b_j,
\]
which finishes the proof.
\end{proof}

The proof of the following main theorem uses results in~\cite{H}.

\begin{thm}\label{thm:optimalmixed subdivision}
Let $\A$ be a circuit in $\Z^n$, let $\{K_0, \dots, K_{k-1}\}$ be a partition of $[n+2]$, and let 
$\bar{\sigma}: [k] \rightarrow K$ be any bijection, where $K$ is any set of representatives of $K_0, \dots, K_{k-1}$. 
Let $B$ the matrix Gale dual to $A$ with coefficients $b_j = (-1)^j \, \det(A(j))$, $j=0,\ldots,n+1$.
Consider the sequence $\mu$ defined by \eqref{eqn:LambdaAndMu} with indices of sign changes
$\ell_{0}<\ell_{1}<\cdots<\ell_{m}$.

Then, there exist matrices $C$ and $H$ in $\R^{n \times (n+2)}$ such that the associated
regular mixed subdivision of $\A_{1}+\cdots+\A_{n}$
contains precisely $m+1$ positively decorated mixed cells $Z_{0},\ldots, Z_{m}$
with Euclidean volumes $\vol(Z_{s})=|\mu_{\ell_{s}}|$ for $s=0,\ldots,m$.

Moreover,  for any real number $t$ with $0 < t < 1$ the matrix  $C_{t}$ in~\eqref{eq:Ct}
has maximal rank, satisfies~\eqref{eq:nonempty}, and  $\bar{\sigma}$ is a strict ordering of $C_{t}$.
Furthermore, there exists $t_{0}>0$ such that for any $t$ satisfying $0<t<t_{0}$,
the bound \eqref{Eq:mainDescartes} is sharp for the associated  Viro polynomial system
 $$\sum_{j=1}^{s} c_{ij}t^{h_{ij}}x^{a_{j}}=0, \; i=1,\ldots,n$$
If, in addition, $m=n$, then we have $\sum_{s=0}^{n}\vol(Z_{s})=n!\, \vol(\A)$.
\end{thm}

\begin{proof}
Since $\mu$ has at most $n+1$ non-zero entries,
we find that 
\[
\sgnvar(\mu) \leq n
\]
with equality if and only if $k=n+2$ and $\mu_{j} \cdot \mu_{j+1}<0$ for $j=0,\ldots, n$.
In this case, the existence of a matrix $C$ as required follows from \cite[Theorem 1.1]{H}.

We now turn to the general case. 
Let $\ell_0<\ell_1<\cdots<\ell_m$ be  the indices of sign changes in $\mu$.
Define
\[
M_s = \bigcup_{i = \ell_{s-1}+1}^{\ell_s} K_i, \quad s = 1, \dots, m
\]
and, in addition,
\[
M_0 = \bigcup_{i = 0}^{\ell_0} K_i,
\qquad \text{ and } \qquad
M_{m+1} = \bigcup_{i = \ell_m+1}^{k-1} K_i,
\]
where $\{K_0, \dots, K_{k-1}\}$ is the desired partition of $[1+n]$.
Set
\[
\alpha_s=\sum_{j \in M_s} b_{j}, \quad s=0,\ldots,m+1.
\]

Assume without loss of generality that $\alpha_0>0$.
Then,
\[
\sgn(\mu_{\ell_{s}})=\sgn(\alpha_{s})=(-1)^{s}, \quad s \in [m+1].
\]
Define,
\[
\beta_{s}=\alpha_{0} + \dots + \alpha_{s}, \quad s \in [m+2]
\] 
so that the sequence $\beta$ relates to $\alpha$ just as $\mu$ relates to $b$.
Actually,
\[
\beta_{s}=\mu_{\ell_s} \text{ for } s \in [m+1],
\]
and $\beta_{m+1}=0$. 
We have that $\sgnvar(\beta)=\sgnvar(\mu) = m$,
which is one less than the number of non-zero entries of $\beta$.
That is, the sequence $\beta$ is as in the special case considered in
the first paragraph of the proof, with $m=n$. 

Choose points $\hat a_0,\ldots,\hat a_{m+1}  \in \R^m$ such that the rows of
\[
\widehat A = \left[\begin{array}{ccc} 1 & \cdots & 1 \\ \hat a_0 &\cdots &   \hat a_{m+1} \end{array}\right]
\]
are a basis for the kernel of the matrix $(\alpha_1, \dots, \alpha_{m+1})$.
By \cite[Theorem 1.1]{H}, there exists a matrix $\widehat C \in \R^{m \times (m+2)}$
such that the system with coefficient matrix $\widehat C$ and support $\widehat{\mathcal A}=\{{\hat a}_0,\ldots,{\hat a}_{m+1}\}$  
has $m+1$ positive solutions, counted with multiplicities.

Let $\widehat P$ denote a Gale dual of the coefficient matrix $\widehat{C}$ which satisfies~\eqref{eq:v}. For $j \in [n+2]$
set $P_{j}=\widehat P_s$ if $j \in M_s$.
Choose any matrix $C \in \R^{n \times (n+2)}$ for which the matrix $P$ with rows $\{P_0, \dots, P_{n+1}\}$ is a Gale dual. 
Applying twice Proposition \ref{prop:Galebijection} we get 
\[
n_\A(C)=n_{\widehat A}(\widehat C)=m+1.
\]
Indeed, setting $p_{j}(y)=\langle P_{j}, (1,y) \rangle$ for $j \in [n+2]$ and  ${\hat p}_{s}(y)=\langle {\widehat P}_{s}, (1,y) \rangle$ for $s \in [m+2]$, we get
\begin{equation}\label{E:same}
\prod_{j \in [n+2]}p_{j}(y)^{b_{j}}=  \prod_{s \in [m+2]} {\hat p}_{s}(y)^{\alpha_{s}},
\end{equation}
and obviously the common domain of positivity $\Delta_{P}$ of the polynomials $p_{j}$ coincides with that of the polynomials ${\hat p}_{s}$.

 Replacing $A$ by a matrix $A_{\sigma}$ whose columns are the elements of $\A$ arranged with respect to a permutation
$\sigma \in \mathfrak{S}_{[n+2]}$ inducing $\bar{\sigma}$, we might assume without loss of generality that
each $K_{j}$ consists of consecutive integer numbers, and that $x<y$ for any $x \in K_{i}$ and $y \in K_{j}$ such that $i<j$. In particular, $0 \in K_{0}$ and $n+1 \in K_{k-1}$.

Let 
\[
\hat h_0< \hat h_1< \cdots < \hat h_{m}
\] 
be a strictly increasing sequence of real numbers, and let 
$t \in (0,1)$ be a parameter. Set 
\[
\widehat{P}_0=(0,t^{\hat h_0}), \quad \widehat{P}_{m+1}=(1,0), \quad \text{ and } \quad \widehat{P}_s=(1,t^{\hat h_{s}}), \quad s=1,\ldots,m.
\] 
Note that the vectors $\widehat{P}_s$, for $s = 0, \dots, m+1$, are contained in an open half plane passing through the origin, 
and  $\det(\widehat{P}_{j}, \widehat{P}_{i})<0$ whenever $i<j$. 
Let $M_0, \dots, M_{m+1}$ as above and define, for $\ell \in M_s$,
\[
P_{\ell}=\widehat{P}_s \quad \text{ and } \quad h_\ell = \hat h_s.
\]
Set 
\[
\hat p_{s}(y)=\langle \widehat{P}_{s}, (1,y) \rangle
\quad \text{ and } \quad
p_{j}(y)=\langle P_{j}, (1,y) \rangle,
\]
where $s \in [m+2]$ and $j \in [n+2]$,
so that \eqref{E:same} holds.
Since $p_{j}(y)=1$ for $j \in M_{m+1}$, the right product in \eqref{E:same} 
can be taken over $[m+1]$ and the left product in \eqref{E:same} can be taken over $M_0 \cup \dots \cup M_{m}$.
Recall that $n_{A}(C)$ is equal to the number of solutions of \eqref{eqn:GaleDualSystem} in $\Delta_{P}=\Delta_{\widehat{P}}$,
and 
\[
\sgn(\alpha_{s}) = \sgn(\beta_s) =(-1)^{s}, \quad s \in [m+1].
\] 
Thus, we may rewrite equation~\eqref{eqn:GaleDualSystem} as an equation in $y$, dependent on a parameter $t$,
(cf.~\eqref{eq:g})
\begin{equation}
g_{+}(y)-g_{-}(y)=0,
\end{equation}
where
\[
g_+(y) = \prod_{\begin{subarray}{c}{s \in [m+1]}\\{s \text{ even}}\end{subarray}} {\hat p}_{s}(y)^{\alpha_{s}}
\quad \text{and} \quad 
g_-(y) = \prod_{\begin{subarray}{c}{s \in [m+1]}\\{s \text{ odd}}\end{subarray}} {\hat p}_{s}(y)^{-\alpha_{s}}.
\]
Then, 
we have the following interlacing property between the exponents of the
monomials of $g_{+}$ and $g_{-}$:
\begin{equation}\label{interlacing}
0< \alpha_{0}<-\alpha_{1}<\alpha_{0}+\alpha_{2}<-(\alpha_{1}+\alpha_{3})<\alpha_{0}+\alpha_{2}+\alpha_{4}<\cdots,
\end{equation}
where the last term is the degree of $g_{+}(y)- g_{-}(y)$:
\[
\deg\big(g_+(y) - g_-(y)\big) = \left\{\begin{array}{ll}
\displaystyle\sum_{\begin{subarray}{c}{s \in [m+1]}\\{s \text{ even}}\end{subarray}}\phantom{-}\alpha_{s} & \text{ if } m \text{ is even,}\\
\displaystyle\sum_{\begin{subarray}{c}{s \in [m+1]}\\{s \text{ odd}}\end{subarray}} -\alpha_{s}& \text{ if } m \text{ is odd.}
\end{array}\right.
\]

Consider $g_{+}$ and $g_{-}$ as generalized polynomials (i.e., exponents are real numbers) in the pair of variables $(y,t)$.
Set $Q_{-1}=(0,0)$ and for $s=0,\ldots,m$ set
$$Q_{s}=\left(\sum_{\begin{subarray}{c}{r\,\leq\, s} \\ {r \text{ even}}\end{subarray}} \alpha_{r}
 \sum_{\begin{subarray}{c}{r \,\leq \,s}\\{r \text{ even}}\end{subarray}}\alpha_{r}\hat h_{r} \right)$$ if 
s is even and $$Q_{s}=\left(\sum_{\begin{subarray}{c}{r\,\leq \,s}\\{r \text{ odd}}\end{subarray}} -\alpha_{r},
\sum_{\begin{subarray}{c}{r\,\leq\, s}\\{r \text{ odd}}\end{subarray}} -\alpha_{r}\hat h_{r}\right)$$ if $s$ is odd.
The Newton polytope of of $g_{+}$ (resp., $g_{-}$) is the convex hull of the points $Q_{s}$ with $s$ even (resp., $s$ odd). 
The Newton polytope $N$ of  $g_{+}-g_{-}$ is the convex hull of the $m+2$ points $Q_{-1}, Q_{0},\ldots,Q_{m}$. 
Note that the abcissa of these points increase as their indices by  \eqref{interlacing}.
The lower part of $N$ is the part of its boundary where some linear function $\langle (u,1), \cdot \rangle$ attains its minimum. Then, there exists an increasing sequence 
 $\hat h_0<\hat h_1< \cdots < \hat h_{m}$ such that the lower part of $N$ consists of the $m+1$ edges
$$E_{s}=\{Q_{s-1}, Q_{s}\} , \; s=0,\ldots,m.$$
This fact is already used in \cite{H}, where a detailed computation is provided. 

Assume without loss of generality that $\hat h_{0}=0$. 
The necessary and sufficient conditions for which $E_{0}$ is a face of $N$ where a linear function $\langle (u,1), \cdot \rangle$ takes its minimum
are
\begin{equation}
\label{regularGale=0}
u= 0 \quad \text{ and } \quad \hat h_{r} >0, \quad r=1,\ldots, m,
\end{equation}
while for $s \geq 1$ the necessary and sufficient conditions for which $E_{s}$ is a face of $N$ where a linear function $\langle (u,1), \cdot \rangle$ takes its minimum are
\begin{equation}\label{regularGale>0}
\left\{\begin{array}{rll}
u+\hat h_{r}\phantom{)}   &<  & 0,  \quad  r=1,\ldots,s, \\
u+\hat h_{r} \phantom{)}   &> &  0 , \quad r=s+1,\ldots, m,\\
\sum_{r=0}^{s}\alpha_{r}(u+\hat h_{r})  &= &  0. 
\end{array}
\right.
\end{equation}
The equality in \eqref{regularGale>0} is obtained from that the linear function takes the same value at the endpoints of the edge $E_{s}$, while the inequalities 
in~\eqref{regularGale>0} are obtained using that, when $s$ is even, $Q_{s-1}$ and $Q_{s}$ are the vertices of the lower parts of the Newton polytopes 
of $g_{-}$ and $g_{+}$, respectively, where the linear function is minimized (ans similary when $s$ is odd by permuting $g_{+}$ and $g_{-}$).

Let $j_{0}, j_{1}, \ldots, j_{m}$ be the integer numbers such that
$M_0=\{0,1,\ldots,j_{0}\}$, $\cup_{r=1}^{s} M_{r}=\{j_0+1,j_{0}+2,\ldots,j_{s}\}$, $s=1,\ldots,m$,
and $M_{m+1}=\{j_m+1, j_{m}+2,\ldots,n+1\}$. Then, 
the vectors $P_{j} \in \R^2$ for $j \in [n+2]$ are Gale dual vectors of the coefficient matrix  of the following Viro polynomial system. 
We assume $a_{n+1}=0$ and, hence, $x^{a_{n+1}}=1$.
\begin{equation}\label{Viro}
\left\{
\begin{array}{ll}
x^{a_{j}}=x^{a_0} & j = 1, \dots, j_0\\
x^{a_{j}}= 1+t^{h_{j}} x^{a_0}
& j = j_0+1, \dots, j_{m}\\
x^{a_{j}}=1 & j = j_m+1, \dots, n.
\end{array}
\right.
\end{equation}
(If $j_{0}=0$, there is no equation $x^{a_{j}}=x^{a_0}$.)  
We claim that the edges $E_{s}$ defined above are in one-to-one correspondence with mixed cells $Z_{s}$ contained in the regular mixed subdivision corresponding to the system \eqref{Viro}.
For $s=0,\ldots,m$, set
\begin{equation}\label{mixed cells}
Z_{s}=\sum_{j = 1}^{j_s} \{a_{0},a_{j}\}+\sum_{j=j_s+1}^{n} \{a_{n+1},a_{j}\}.
\end{equation}

Note that each $Z_{s}$ is positively decorated by the coefficient matrix of~\eqref{Viro}.
Viewing~\eqref{Viro} as a generalized polynomial system in the $n+1$ variables $(x,t)$,
we get $n$ equations whose associated supports 
are the following triangles and (possibly) line segments in 
$\R^n \times \R$.
\begin{equation}\label{triangles}
\left\{\begin{array}{ll}
S_{j}=\{(a_{j},0),(a_{0},0)\}, &
\text{ if }j = 1, \dots, j_0,\\
S_{j}=\{(a_{j},0), (0,0), (a_{0},h_{j})\}, &
\text{ if }j = j_0+1, \dots, j_m,\\
S_{j}=\{(a_{j},0), (0,0)\}, &
\text{ if }j= j_m+1, \dots, n.
\end{array}\right.
\end{equation}

Consider the case $s >0$ and assume that \eqref{regularGale>0} is satisfied.
Then,
\begin{equation}\label{regularGale>0bis}
\left\{\begin{array}{rll}
u+h_{j}\phantom{)}  & <  &0,  \quad j = j_0+1, \dots, j_s \\
u+h_{j}\phantom{)}   &>  &0 , \quad j = j_s+1, \dots, j_m \\
\sum_{j  = 0}^{j_s} b_{j}(u+h_{j}) & =  &0. 
\end{array}
\right.
\end{equation}
Using \eqref{regularGale>0bis}, we see that if $z \in \R^n$ satisfies the conditions
\begin{equation}
\label{The conditions}
\left\{\begin{array}{ll}
\langle z, a_{j} \rangle=u, & \mbox{for } {j = 0, \dots, j_0} \\
\langle z, a_{j} \rangle=u+h_{{j}} & \mbox{for } j  = j_0+1, \dots, j_s\\
\langle z, a_{j} \rangle=0 &\mbox{for } j = j_s+1, \dots, n
\end{array}
\right.
\end{equation}
then, writing $\psi$ for the linear function given by the scalar product with $(z,1)$,
we get
 \begin{equation}
\left\{\begin{array}{ll}
\psi(a_{j},0) =\psi (a_{0},0) &  \mbox{for } j = 0, \dots, j_0 \\
\psi(a_{j},0)=\psi(a_{0},h_{j})<\psi(0,0)=0 
& \mbox{for }  j = j_0+1, \dots, j_s\\
\psi(a_{j},0)=\psi(0,0)=0< \psi(a_{0},h_{j})
& \mbox{for }j = j_s+1, \dots, j_m\\
\psi(a_{j},0)=\psi(0,0)
& \mbox{for }  j = j_m+1, \dots, n+1
 \end{array}
 \right.
 \end{equation}
which implies that $Z_{s}$ is a mixed cell of the mixed subdivision associated to the system \eqref{Viro}.
We now show that the existence of $z \in \R^n$ satisfying \eqref{The conditions}.
Since $h_{j}=h_{0}=0$ for all $j \in M_{0}$,
the first two lines of \eqref{The conditions} are equivalent to
\[
\langle z, a_{j} \rangle=u+h_{{j}}, \quad \text{for } j = 0, \dots, j_s.
\]
The system \eqref{The conditions} consists of $n+1$ equations.
Let $j(s)$ be any element in $\{j_0+1, \dots, j_s\}$. 
Forgetting the equation corresponding to $j=j(s)$ the system \eqref{The conditions} has exactly one solution in $z$ since the $n$ vectors $a_{j}$ for 
$j \in \{0,\ldots,n\} \setminus \{j(s)\}$ are linearly independent. Then, 
\begin{align*}
0&=\bigg\langle z, \sum_{j=0}^{n}b_{j}a_{j}\bigg\rangle
 = \sum_{j=0}^{n}b_{j}\langle z, a_{j}\rangle=\sum_{j=0}^{j_{s}}b_{j}\langle z, a_{j}\rangle\\
 & =b_{j(s)}\big(\<z, a_{j(s)}\>  - u - h_{j(s)}\big) + \sum_{j=0}^{j_s}b_{\ell}(u+h_{\ell})\\
 & = b_{j(s)}\big(\<z, a_{j(s)}\>  - u - h_{j(s)}\big),
\end{align*}
where in the last equality we used \eqref{regularGale>0bis}. 
We conclude that $\langle z, a_{j(s)} \rangle=u+h_{{j(s)}}$, 
and thus $z$ is a solution of the whole system \eqref{The conditions}.

Consider now the case $s=0$, and assume \eqref{regularGale=0} is satisfied.
Then we get $u=0$ and $h_{j}  >  0$ for all $j = j_0+1, \dots, j_m$.
Writing $\psi$ for the linear function given by the scalar product with $(0,1)$ we get
\begin{equation}
\left\{\begin{array}{ll}
\psi(a_{j},0) =\psi (a_{0},0) & \mbox{for } j = 0, \dots, j_0, \\
\psi(a_{j},0)=\psi(0,0)=0< \psi(a_{0},h_{j})
& \mbox{for } j = j_0+1, \dots, j_m\\
\psi(a_{j},0)=\psi(0,0)
& \mbox{for }  j = j_m+1, \dots, n+1.
 \end{array}
 \right.
 \end{equation}
which shows that $Z_{0}$ is a mixed cell of the mixed subdivision associated to the system \eqref{Viro}. 
Using Lemma \ref{Stu} we get the assertion about the number of positive solutions of the Viro polynomial system. We also obtain that there cannot be more 
positively decorated mixed cells since otherwise this Viro polynomial system would have (for $t>0$ small enough) more than
$m+1$ positive solutions, which would contradict the upper bound \eqref{Eq:mainDescartes}.

Lemma~\ref{lem:MinkowskiSumIsLambdaSum} shows that the Euclidean volume of $Z_{s}$ is the absolute value of $\mu_{\ell_{s}}$.
Finally, we have that $m=n$ if and only if $k=n$ and $\sgnvar(\mu) = n$.
Assuming $\mu_{0}>0$, we thus have $(-1)^j \mu_{j} >0$, which implies
that $(-1)^j b_{j} >0$ for $j=0,\ldots,n$. Therefore, we have
\begin{align*}
\sum_{j=0}^{n} |\mu_{j}| &=\sum_{j=0}^{n}(-1)^j \mu_{j}=b_{0}-(b_{0}+b_{1})
+(b_{0}+b_{1}+b_{2})-\ldots\\
& 
= \left\{\begin{array}{ll}
\displaystyle\phantom{-}\sum_{j \text{ even}} b_j& \text{ if } n \text{ is even}\\
\displaystyle-\sum_{j \text{ odd}} b_{j}& \text{ if } n \text{ is odd}
\end{array}\right.
\end{align*}
which is equal to $\vol_{\Z}(\A)=n!\, \vol(\A)$.
\end{proof}

\subsection{The configuration space in case $n=2$}
  We describe the partitioning of the moduli space of circuits which are the support of bivariate polynomial systems
into equivalence classes describing the maximal number of positive solutions.
For any circuit $\A$, the maximum value of $n_\A(C)$ (when it is finite) coincides by Theorem~\ref{thm:optimalmixed subdivision}
with the maximum upper bound provided by Theorem~\ref{thm:refines}. It follows that this maximum equals either  $3$ or $2$.
We depict in Figure~\ref{fig:ConfigSpaceDim2} the configuration space of circuits $\A$.
\begin{figure}[t]
\includegraphics[height=50mm]{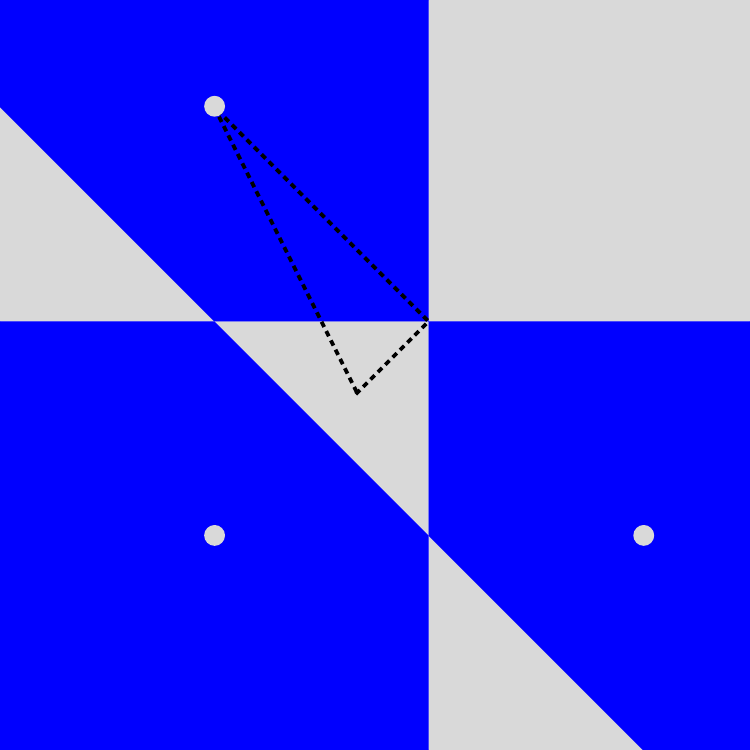}
\hspace{10mm}
\includegraphics[height=50mm]{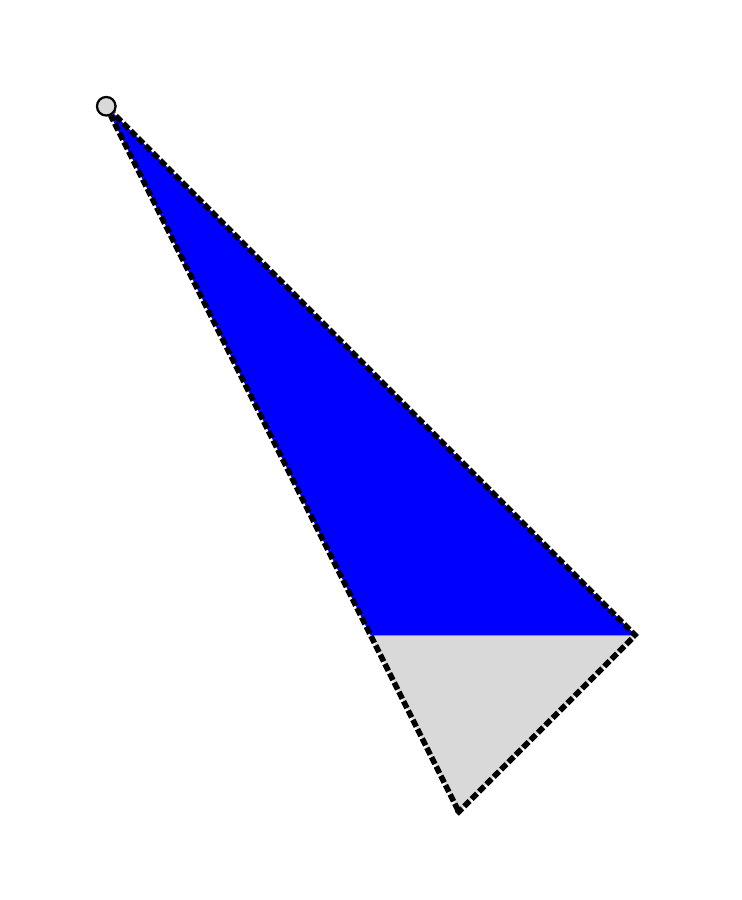}
\caption{The configuration space in dimension two.}
\label{fig:ConfigSpaceDim2}
\end{figure}
 Each such support set $\A$ is determined (up to affine transformation)
by a Gale dual $B = (b_0, \,  b_1, \, b_2, \,b_3 )^\top$ of the associated matrix $A$.
Recall that all $b_j \neq 0$ because $\A$ is a circuit.

There is a real polynomial system supported on $A$ with three positive solutions 
if and only if there exists an ordering of the indices $\{0,1,2,3\}$ (i.e., a permutation $\sigma$)
such that
\[
b_{\sigma_{0}}(b_{\sigma_{0}} + b_{\sigma_{1}}) < 0 
\quad 
\text{ and }
\quad
b_{\sigma_{0}} (b_{\sigma_{0}} + b_{\sigma_{1}}+ b_{\sigma_{2}}) > 0.
\]
As $|\mathfrak{S}_4| = 24$, the loci of support sets which admits three positive solutions
is covered by 24 open polyhedral cones. 
But Theorem~\ref{thm:refines} implies that in this case $\sgnvar(b_{\sigma_{0}}, b_{\sigma_{1}}, b_{\sigma_{2}}, b_{\sigma_{3}})=3$
for some $\sigma \in \mathfrak{S}_4$.
Since $\sum_{j=0}^{3} b_j=0$ not all coefficients have the same sign. Then, there are two positive and two negative coefficients in this sequence.
This implies that a necessary condition to ensure a sparse polynomial system with $3$ positive solutions is simply
\begin{equation}\label{eq:bes}
 b_0 \, b_1\, b_2 \, b_3 > 0,\end{equation}
 but we need to moreover ensure that the sign variation of the corresponding sequence $\mu$ equals $2$.
 It turns out that except for the special case $(b_{0},b_{1},b_{2},b_{3})=(-1,-1,1,1)$ (up to permutation $\sigma$ and rescaling)
  the condition~\eqref{eq:bes} is sufficient to ensure a sparse polynomial system with $3$ positive solutions.

Assume that $b_3=1$. As $b_0 = -b_1 -b_2 - b_3$,  we are left with two coordinates $(b_1, b_2)$.
The loci of support sets which admit three positive solutions consists of the open set
\[
b_1\,  b_2\,  (1+b_1+b_2) < 0
\]
with exception of the three points $(-1,-1), (1,-1)$  and  $(-1, 1)$.
This region is shown in blue in the left picture in Figure~\ref{fig:ConfigSpaceDim2}.
The three exceptional points correspond to support sets $\A$ such that 
some sum~\eqref{eq:ls} vanishes.

 Dividing by the maximum of the absolute values of $b_0, \dots, b_3$, they will all lie in the interval $[-1,1]$. 
Therefore, up to permutation, each support set $\A$ admits
a representative such that
\[
-1 \leq b_0 \leq b_1 \leq b_2 \leq b_3 = 1.
\]
Note that both $b_0 = -1 -b_1-b_2$ and $b_1$ need to be negative, so~\eqref{eq:bes} reduces to the simpler condition that $b_2 > 0$.
The \emph{fundamental region} is marked by a dotted line in the left picture in
Figure~\ref{fig:ConfigSpaceDim2},
and has been magnified in the right picture in that figure.

\providecommand{\bysame}{\leavevmode\hbox to3em{\hrulefill}\thinspace}
\providecommand{\MR}{\relax\ifhmode\unskip\space\fi MR }
\providecommand{\MRhref}[2]{%
  \href{http://www.ams.org/mathscinet-getitem?mr=#1}{#2}
}
\providecommand{\href}[2]{#2}


\begin{thebibliography}{1}


\bibitem{BD} F. Bihan, A. Dickenstein
\emph{Descartes' Rule of Signs for Polynomial systems supported on circuits}, Int. Math. Res. Notices \textbf{22} (2017), 6867--6893.

\bibitem{BDG} F. Bihan, A. Dickenstein and M. Giaroli, \emph{Lower bounds for positive roots and regions of multistationarity 
in chemical reaction networks},  J. Algebra \textbf{542} (2020), 367--411.



\bibitem{BS07}
F.~Bihan and F.~Sottile, \emph{{New fewnomial upper bounds from {G}ale dual
  polynomial systems}}, Mosc. Math. J. \textbf{7}, no.~3 (2007), 387--407.

\bibitem{H}
B. El Hilany,
\emph{Characterization of circuits supporting polynomial systems with the maximal number of positive solutions}, 
Discrete Comput. Geom. (2017), Vol. 58, No. 2, pp. 355--370.


\bibitem{GKZ} I. M. Gelfand, M. M. Kapranov, A. V. Zelevinsky,
\emph{Discriminants, resultants, and multidimensional determinants},
Mathematics: Theory \& Applications. Birkhauser Boston, Inc., Boston, MA, 1994.

\bibitem{HRS}
B. Huber, J. Rambau, F. Santos, 
\emph{The Cayley trick, lifting subdivisions and the Bohne-Dress theorem on zonotopal tilings,}
J. Eur. Math. Soc. (JEMS) \textbf{2}, no. 2 (2000), 179--198.

\bibitem{ir96}
I.~Itenberg and M.~F. Roy, \emph{{Multivariate Descartes' rule}}, Beitr.
Algebra Geom. \textbf{37}, no.~2 (1996), 337--346.

\bibitem{Ph} G. Phillips, \emph{Interpolation and Approximation by Polynomials}, CMS books in Mathematics,
Springer, 2003.

\bibitem{Pinkus} A. Pinkus, \emph{Totally positive matrices},
Cambridge Tracts in Mathematics \textbf{181}, Cambridge University Press,  2010.

\bibitem{P-S}
G. P\'olya and G. Szeg{\H{o}},
\emph{Problems and theorems in analysis, Vol. II}, 
Springer-Verlag, New York, 1976.

\bibitem{Stu} B. Sturmfels, 
\emph{Viro's theorem for complete intersections},
Ann. Scuola Norm. Sup. Pisa Cl. Sci. \textbf{4}, vol. 21, no. 3 (1994), 377--386.

\bibitem{Stu-Cayley} B. Sturmfels,
\emph{On the Newton polytope of the resultant,} J. Algebr. Combinatorics \textbf{3}, 207--236 (1994).

\end{thebibliography}
\end{document}